\documentclass[a4paper]{article}
\pdfoutput=1
\usepackage{mathrsfs}
\usepackage{}
\usepackage{amsfonts}
\usepackage{mflogo,texnames}
\usepackage{amsmath,amssymb,amsthm}
\usepackage{color}
\usepackage[colorlinks=true]{hyperref}

\numberwithin{equation}{section}

\newtheorem{theorem}{Theorem}[section]
\newtheorem{lemma}{Lemma}[section]
\newtheorem{proposition}{Proposition}[section]
\newtheorem{definition}{Definition}[section]
\newtheorem{corollary}{Corollary}[section]
\newtheorem{remark}{Remark}[section]

\begin{document}

\title{\textbf{New lower bounds on the radius of spatial analyticity for the KdV equation}}

\bigskip

\author{
$^a$Jianhua Huang \quad
$^{ab}$Ming Wang\thanks{Corresponding author. \emph{Email: mwangcug@outlook.com(M. Wang), jhhuang32@nudt.edu.cn(J. Huang)}
$^a$College of Science, National University of Defense Technology, Changsha, 410073, PR China; $^b$School of Mathematics and Physics, China University of Geosciences, Wuhan, 430074,  China.}
}

\maketitle

\begin{abstract}
The radius of spatial analyticity for solutions of the KdV equation is studied. It is shown that the analyticity radius does not decay faster than $t^{-1/4}$ as time $t$ goes to infinity. This improves the works [Selberg, da Silva, Lower bounds on the radius of spatial analyticity for the KdV equation, Annales Henri Poincar\'{e}, 2017, 18(3): 1009-1023] and [Tesfahun, Asymptotic lower bound for the radius of spatial analtyicity to solutions of KdV equation, arXiv preprint arXiv:1707.07810, 2017]. Our strategy mainly relies on a higher order almost conservation law in Gevrey spaces, which is inspired by the $I-$method.
\end{abstract}

\medskip



\section{Introduction}
In this paper, we are concerned with the Cauchy problem for the Korteweg-de Vries (KdV) equation
\begin{align}\label{equ-1}
u_t+u_{xxx}+uu_x=0, \quad t,x\in \mathbb{R}, \quad \quad u(0,x)=u_0(x).
\end{align}
Here, the unknown function $u(t,x)$ and the datum $u_0(x)$ are real-valued. The KdV equation models the unidirectional propagation of small-amplitude long waves in nonlinear dispersive systems. The ill-posedness and well-posedness of the KdV equation in Sobolev spaces $H^s$ have been extensively studied. For instance, Christ,  Colliander  and  Tao \cite{Tao-ajm} showed that the equation \eqref{equ-1} is ill-posed in $H^s(\mathbb{R})$ for $s<-\frac{3}{4}$. Kenig,  Ponce  and  Vega \cite{KPV96} proved the local well-posedness in $H^s(\mathbb{R})$ for $s>-\frac{3}{4}$. With the same range of $s$, the global well-posedness were obtained by Colliander, Keel, Staffilani, Takaoka and Tao in \cite{Tao}. In the critical case $s=-\frac{3}{4}$, the KdV equation is globally well-posed. This is shown by Guo\cite{guo} and Kishimoto\cite{Kis}  independently.

The linear KdV equation, also called the Airy equation, does not have a global smoothing effect. Precisely, it is only expected that $e^{-t\partial_x^3}u_0 (t\neq0)$ belongs to $H^s(\mathbb{R})$ for a general datum $u_0$ belonging to $H^s(\mathbb{R})$. Thus, in principle, the solution of \eqref{equ-1} belongs to at most $H^s(\mathbb{R})$ in general if $u_0$ belongs to $H^s(\mathbb{R})$. But some interesting things happen if  some further restrictions are imposed on the datum. In fact, Kato and Ogawa \cite{KO} showed that if the datum $u_0$ belongs to $H^{s}(\mathbb{R})(s>-\frac{3}{4})$ and satisfies
$$
\sum_{k=0}^\infty\frac{A_0^k}{k!}\|(x\partial_x)^ku_0\|_{H^s}<\infty
$$
for some positive constant $A_0$, then the solution of \eqref{equ-1} is  analytic in both space and time variable. As a direct corollary, if $u_0$ is the Dirac measure at the origin, then the solution of \eqref{equ-1} is analytic. Moreover, Tarama \cite{T04} proved the following result: If $u_0$ belongs to $L^2(\mathbb{R})$ and satisfies
$$
\int_{-\infty}^{\infty}(1+|x|)|u_0(x)|\,\mathrm dx+\int_{0}^{\infty}e^{\delta |x|^{\frac{1}{2}}}|u_0(x)| \,\mathrm dx<\infty
$$
for some positive constant $\delta$, then the solution of \eqref{equ-1} is analytic in spatial variable $x$ for any $t>0$. Tarama's result implies that, roughly speaking, the rapid decay of the datum implies the spatial analyticity of the solution for the KdV equation. The phenomenon was investigated by Rubkin \cite{R13} in a more general framework. It is proved in \cite{R13} that, if $u_0=\mathcal {O}(e^{-c|x|^\alpha})(x\rightarrow +\infty)$ and satisfies some other slight restrictions, then the solution of \eqref{equ-1} is
\begin{description}
  \item[(a)] analytic in $x$ on the whole plane when $\alpha>\frac{1}{2}$,
  \item[(b)] analytic in a strip around the real line when $\alpha=\frac{1}{2}$,
  \item[(c)] Gevrey-regular if $\alpha<\frac{1}{2}$.
\end{description}
Motivated by these works, it is interesting to study the well-posedness for the KdV equation in analytic function spaces.

A nice choice of the analytic function space is the Gevrey space $G^\sigma(\mathbb{R})(\sigma>0)$, consisting of functions such that
$$
\|f\|_{G^\sigma(\mathbb{R})}:=\|e^{\sigma|\xi|}\widehat{f}(\xi)\|_{L^2(\mathbb{R})}<\infty,
$$
where $\widehat{f}(\xi)$ denotes the Fourier transform of $f$. In fact, according to the Paley-Wiener Theorem (see e.g. \cite{Ka}), a function belongs to $G^\sigma$ if and only if it can be extended to an analytic function on the strip
$$
S_\sigma:=\{z\in \mathbb{C}: |\textmd{Im } z|<\sigma\}.
$$
The local well-posedness of the KdV equation in $G^\sigma$ has been studied by several mathematicians.  Gruji\'{c} and Kalisch \cite{GK} showed that, if the datum $u_0$ belongs to $G^{\sigma_0}$ for some $\sigma_0>0$, then the KdV equation \eqref{equ-1} has a unique solution $u\in C[-T,T; G^{\sigma_0}]$ with a lifespan $T$ depending on $\|u_0\|_{G^{\sigma_0}}$. Similar results for the periodic KdV equation are proved by Hannah,  Himonas and Petronilho \cite{H11, H12} and Li \cite{L}. The work by Gruji\'{c} and Kalisch \cite{GK} improved the earlier results of  Hayashi \cite{H91, H-duke}, where the analyticity radius $\sigma(t)$ of local solution may depend on $t$. The local well-posedness in \cite{GHH, GK, H11, H12, L} shows that, for short times, the KdV equation persists the uniform radius of spatial analyticity as time progresses.

Now we turn to the global well-posedness. In Sobolev spaces $H^s$, as mentioned above, the study on the global well-posedness of the KdV equation is more or less complete. However, in analytic function spaces, the global well-posedness of the KdV equation is still open, mainly due to the lack of conversation law. In other words, it is not known whether $u(t)\in G^{\sigma_0}$ for all $t>0$ if $u_0\in G^{\sigma_0}$,  where $u(t)$ is the solution of the KdV equation \eqref{equ-1}. But one can ask the following question instead: For what kind of function $\sigma(t)$ such that $u(t)$ belongs to $G^{\sigma(t)}$ for all $t>0$\footnote{By the embedding $G^\sigma \hookrightarrow G^{\sigma'}$ for $\sigma>\sigma'$, the function $\sigma(t)$ is necessarily less or equal to $\sigma_0$.}?

In the sequel, we recall some progresses on the problem. With the aid of Liapunov functions with a parameter, Kato and Masuda showed \cite[Theorem 2, p. 459]{KM} that, for every $T>0$ fixed, there exists $r>0$ such that $\sigma(t)\geq r$ for $t\in [0,T]$. In particular, the result implies that the solution of the KdV equation is analytic on some strip  at any time. Bona and Gruji\'{c} gave an explicit lower bound of the uniform radius of analyticity by a Gevrey-class approach. In fact, it is shown \cite[Theorem 11 and Remark 12, p. 355]{BG} that $\sigma(t)\geq e^{-ct^2}$ for large $t$, where $c$ depends on the Gevrey norm of the datum. Later, Bona, Gruji\'{c} and Kalisch improved the exponential decay bound to an algebraic lower bound: $\sigma(t)\geq t^{-12}$ for large $t$, see \cite[Corollary 2, p. 795]{BGK}. More recently, Selberg and Silva \cite{SdS} obtained a further refinement: $\sigma(t)\geq t^{-\frac{4}{3}-\varepsilon}$ for large $t$, where $\varepsilon$ is an arbitrary positive number. The strategy in \cite{SdS} is as follows:
\begin{description}
  \item[(1)] Prove a local well-posedness by contraction mapping principle in $G^\sigma$ with a lifespan $\delta>0$;
  \item[(2)] Establish an almost conservation law in $G^\sigma$, namely\footnote{In fact, by letting $\sigma$ go to $0$ in \eqref{equ-intro-1}, one obtained the $L^2$ conservation law of the KdV equation.}
\begin{align}\label{equ-intro-1}
\|u(\delta)\|^2_{G^\sigma}\leq \|u_0\|^2_{G^\sigma}+ C\sigma^{\frac{3}{4}-\varepsilon}\|u_0\|^2_{G^\sigma};
\end{align}
  \item[(3)] By shrinking $\sigma$ gradually, they used repeatedly the local well-posedness and the almost conservation law on the intervals $[0, \delta], [\delta, 2\delta], \cdots$, and obtained a global bound of solution on $[0,T]$, with $T$ arbitrarily large.
\end{description}
In a paper \cite{T17b} on arxiv, Tesfahun removed the $\varepsilon$ exponent in the conservation law \eqref{equ-intro-1}, via spacetime dyadic bilinear estimates associated with the KdV equation. This leads to the following improvement: $\sigma(t)\geq t^{-\frac{4}{3}}$ for large $t$. In this paper, we are able to show that $\sigma(t)\geq t^{-\frac{1}{4}}$ for large $t$. The precise statement is as follows.

\begin{theorem}\label{thm1}
Let $\sigma_0>0$ and $u_0 \in G^{\sigma_0}$. Then the KdV equation \eqref{equ-1} has a unique smooth solution $u$ such that
$$
u(t)\in G^{\sigma(t)}, \quad t\in \mathbb{R}
$$
with the radius of analyticity $\sigma(t)$ satisfying the lower bound
$$
\sigma(t)\geq c|t|^{-\frac{1}{4}} \quad  \textmd{as } |t|\rightarrow \infty,
$$
where $c$ is a constant depending on $\|u_0\|_{G^{\sigma_0}}$ and $\sigma_0$.
\end{theorem}

The proof of Theorem \ref{thm1} is based on a simple observation: If one can prove the almost conservation law
\begin{align}\label{equ-intro-2}
\|u(\delta)\|^2_{G^\sigma}\leq \|u_0\|^2_{G^\sigma}+ C(\|u_0\|^2_{G^\sigma})\sigma^{\alpha}
\end{align}
with a larger $\alpha>0$, then one obtains a better lower bound of $\sigma(t)$ by the strategy in \cite{SdS}. To this end, inspired by the $I-$method in \cite{Tao}, we define the modified energies $E^2_I(t), \cdots, E^4_I(t)$ (see Section 2 for definitions) in Gevrey space $G^\sigma$, and prove that
\begin{align}\label{equ-intro-3}
E^4_I(t)\textmd{ is comparable with }E^2_I(t)\textmd{ for all }t\in \mathbb{R}\textmd{ when }E^2_I\textmd{ is small},
\end{align}
\begin{align}\label{equ-intro-4}
|E^4_I(\delta)-E^4_I(0)|\leq  C\|u_0\|^5_{G^\sigma}\sigma^4.
\end{align}
Combining \eqref{equ-intro-3} and \eqref{equ-intro-4} we find that \eqref{equ-intro-2}  holds with $\alpha=4$ for small $\|u_0\|_{G^\sigma}$. The smallness can be removed by a scaling. In a word,  this leads a better lower bound $\sigma(t)\geq c|t|^{-\frac{1}{4}}$.

We do not believe that the lower bound in Theorem \ref{thm1} is optimal. In fact, it is probably can be improved by introducing further modified energies $E^5_I(t), \cdots$ in the scheme as in \cite{Tao}.


Finally, we mention some references devoted to the uniform radius of analyticity for other partial differential equations. We refer the readers to \cite{BGK10, HKS, ST17} for generalized KdV equations, to \cite{BGK06, CAN, T17a} for Schr\"{o}dinger equations, to \cite{KV09, KV11-a, KV11-b} for Euler equations,  to \cite{GT, P, ST15} for Klein-Gordon equations, and to \cite{GGT} for the cubic Szeg\H{o} equation.

\section{Preliminaries}

\subsection{Local well posedness}
First, we introduce some function spaces used in this paper. For $s, b \in \mathbb{R}$, we use $X^{s,b}(\mathbb{R}^2)$ to denote the Bourgain space defined by the norm
$$
\|f\|_{X^{s,b}(\mathbb{R}^2)}:=\|(1+|\xi|)^s(1+|\tau-\xi^3|)^b\widehat{f}(\xi,\tau)\|_{L^2(\mathbb{R}^2)},
$$
where $\widehat{f}(\xi,\tau)$ denotes the space-time Fourier transform of $f(x,t)$:
$$
\widehat{f}(\xi,\tau) = \int_{\mathbb{R}^2}e^{-i(t\tau+x\xi)}f(x,t)\,\mathrm dx\mathrm dt.
$$
Replacing $(1+|\xi|)^s$ by $e^{\sigma|\xi|}$ in the norm of the Bourgain space,  we obtain a Gevrey type Bourgain space $G^{\sigma,b}(\mathbb{R}^2)$ defined by the norm
$$
\|f\|_{G^{\sigma,s,b}(\mathbb{R}^2)}:=\|e^{\sigma|\xi|}(1+|\tau-\xi^3|)^b\widehat{f}(\xi,\tau)\|_{L^2(\mathbb{R}^2)}.
$$
Also, for $\delta>0$, we use $X^{s,b}_\delta, G^{\sigma,b}_\delta$ to denote the restrictions of $X^{s,b}$ and $G^{\sigma,b}$ to $\mathbb{R}\times (-\delta, \delta)$, respectively. More precisely,  $X^{s,b}_\delta$ and $ G^{\sigma,b}_\delta$ are defined by the norms as follows:
$$
\|f\|_{X^{s,b}_\delta} = \inf \{\|g\|_{X^{s,b}}: g=f \textmd{ on } \mathbb{R}\times (-\delta, \delta)\},
$$
$$
\|f\|_{G^{s,b}_\delta} = \inf \{\|g\|_{G^{s,b}}: g=f \textmd{ on } \mathbb{R}\times (-\delta, \delta)\}.
$$

Next, we give a local well posedness result for the KdV equation.

According to Corollary 2.7 in \cite{KPV96}, for $b'\in (\frac{1}{2}, \frac{3}{4}]$ and $b\in (\frac{1}{2},b']$, there exists a positive constant $c=c(b,b')$ such that
\begin{align}\label{equ-bi-1}
\|\partial_x(uv)\|_{X^{0,b'-1}}\leq c\|u\|_{X^{0,b}}\|v\|_{X^{0,b}}.
\end{align}
Using the obvious inequality $e^{\sigma|\xi|}\leq e^{\sigma|\xi_1|}e^{\sigma|\xi_2|}$, $\xi=\xi_1+\xi_2$, we deduce from \eqref{equ-bi-1} that with the same range of $b$ and $b'$
\begin{align}\label{equ-bi-2}
\|\partial_x(uv)\|_{G^{\sigma,b'-1}}\leq c\|u\|_{G^{\sigma,b}}\|v\|_{G^{\sigma,b}}.
\end{align}
Applying the bilinear estimate \eqref{equ-bi-2} with $b'=\frac{3}{4}$, and using the contraction mapping principle (or following the proof of \cite[Theorem 1]{SdS}), we obtain the following result.

\begin{proposition}[local well posedness]\label{prop-loc}
Let $\sigma>0$ and $b\in (\frac{1}{2}, \frac{3}{4})$. Then, for any $u_0\in G^\sigma(\mathbb{R})$, there exists a time $\delta >0$ given by
\begin{align}\label{equ-life}
\delta = \frac{c_0}{(1+\|u_0\|_{G^\sigma})^{\frac{1}{\frac{3}{4}-b}}}
\end{align}
and a unique solution $u$ of \eqref{equ-1} such that
\begin{align}\label{equ-local-bound}
\|u\|_{G^{\sigma,b}_\delta}\leq C\|u_0\|_{G^\sigma},
\end{align}
where the constants $C, c_0$ depend only on $b$.
Moreover, the solution map $u_0 \mapsto u(t)$ is continuous from $G^\sigma$ to $G^\sigma$ for every $t\in [-\delta,\delta]$.
\end{proposition}

Finally, we state a multi-linear estimates will be used later.

\begin{lemma}\label{lem1}
Let $\sigma\geq 0, \delta>0, -\frac{1}{2}<b'<-\frac{1}{3}$ and $b>\frac{1}{2}$. Let $|D|$ be the Fourier multiplier with symbol $|\xi|$. Then there exists a constant $C=C(b,b')$ such that
\begin{align}\label{equ-mult-0}
\left\||D|\prod_{i=1}^4u_i\right\|_{X_\delta^{0,b'}}\leq C \prod_{i=1}^4\|u_i\|_{X_\delta^{0,b}}.
\end{align}
\end{lemma}
\begin{proof}
Gr\"{u}nrock  \cite[Theorem 1]{G} proved the following multi-linear estimates: If $-\frac{1}{2}<b'<-\frac{1}{3}$, $b>\frac{1}{2}$, then for some $C=C(b,b')$
\begin{align}\label{equ-mult-1}
\left\|\partial_x\prod_{i=1}^4u_i\right\|_{X^{0,b'}}\leq C \prod_{i=1}^4\|u_i\|_{X^{0,b}}.
\end{align}
By Plancherel's theorem, $\|\partial_x\prod_{i=1}^4u_i\|_{X^{0,b'}}=\||D|\prod_{i=1}^4u_i\|_{X^{0,b'}}$. The desired bound \eqref{equ-mult-0} follows from \eqref{equ-mult-1} in a standard way, see e.g. \cite[Corollary 1]{wangm}.
\end{proof}

\subsection{Multi-linear forms for KdV}

In this section, we borrow some known results from \cite{T04} on multi-linear  forms for the KdV equation.

\begin{definition}
A $k$-multiplier is a function $m: \mathbb{R}^k \mapsto \mathbb{C}$. A $k$-multiplier is symmetric if $m(\xi_1,\xi_2,\cdots,\xi_k)=m(\sigma(\xi_1,\xi_2,\cdots,\xi_k))$ for all $\sigma\in S_k$, the group of all permutations on $k$ objects. The symmetrization of a $k-$multiplier is the multiplier
$$
[m]_{sym}(\xi_1,\xi_2,\cdots,\xi_k)=\frac{1}{k!}\sum_{\sigma\in S_k} m(\sigma(\xi_1,\xi_2,\cdots,\xi_k)).
$$
\end{definition}

\begin{definition}
A $k$-multiplier generates a $k$-linear functional or $k$-form acting on $k$ functions $u_1,u_2,\cdots, u_k$,
$$
\Lambda_k(m;u_1,u_2,\cdots, u_k)=\int_{\xi_1+\xi_2+\cdots+\xi_k=0}m(\xi_1,\xi_2,\cdots,\xi_k)\widehat{u_1}(\xi_1)\widehat{u_2}(\xi_2)\cdots\widehat{u_k}(\xi_k).
$$
In particular, if $u_1=u_2=\cdots=u_k=u$ we write $\Lambda_k(m)=\Lambda_k(m;\underbrace{u,u,\cdots,u}_{k \, \, times})$ for brevity.
\end{definition}

If $m$ is symmetric, then $\Lambda_k(m)$ is a symmetric $k$-linear functional. The symmetry is important in the following discussion. To see this, we give a Fourier proof of the fact
\begin{align*}
\int_{\mathbb{R}}u^{k}u_x \,\mathrm dx=0, \quad k\in \mathbb{N}, u\in \mathscr{S}.
\end{align*}
Indeed, by the Plancherel's theorem, we write
\begin{align*}
\int_{\mathbb{R}}u^{k}u_x \,\mathrm dx&=\int_{\mathbb{R}}i\xi_1\widehat{u}(\xi_1)\widehat{u^k}(-\xi_1)\,\mathrm d\xi_1\\
&=\int_{\xi_1+\xi_2+\cdots+\xi_{k+1}=0}i\xi_1\widehat{u}(\xi_1)\widehat{u}(\xi_2)\cdots \widehat{u}(\xi_k)\\
&=\int_{\xi_1+\xi_2+\cdots+\xi_{k+1}=0}i\xi_j\widehat{u}(\xi_1)\widehat{u}(\xi_2)\cdots \widehat{u}(\xi_k)  \,\, \,\,\,\,(j=2,\cdots, k+1)\\
&=\int_{\xi_1+\xi_2+\cdots+\xi_{k+1}=0}i\frac{\xi_1+\xi_2+\cdots+\xi_{k+1}}{k+1}\widehat{u}(\xi_1)\widehat{u}(\xi_2)\cdots \widehat{u}(\xi_k)=0.
\end{align*}

\begin{proposition}\emph{\textbf{\cite[Proposition 1]{T04}}}\label{prop-1}
Suppose $u$ satisfies the KdV equation \eqref{equ-1} and that $m$ is a symmetric $k$-multiplier. Then
\begin{align}\label{equ-modi-0}
\frac{d}{dt}\Lambda_k(m)=\Lambda_k(m\alpha_k)-i\frac{k}{2}\Lambda_k(m(\xi_1,\cdots,\xi_{k-1},\xi_k+\xi_{k+1})\{\xi_k+\xi_{k+1}\}),
\end{align}
where
\begin{align}\label{equ-alpha-k}
\alpha_k=i(\xi_1^3+\cdots+\xi_k^3).
\end{align}
\end{proposition}
\begin{remark}\label{rem-modi}
Note that \eqref{equ-modi-0} still holds if the $k+1$-multiplier of the second term is symmetrized.
\end{remark}

Let $m: \mathbb{R} \mapsto \mathbb{R}$ be an arbitrary even $\mathbb{R}$-valued $1$-multiplier. Define the associated operator by
$$
\widehat{If}(\xi)=m(\xi)\widehat{f}(\xi).
$$
Define the modified energy $E^2_I(t)$ by
\begin{align}\label{equ-energy-2}
E^2_I(t)=\|Iu(t)\|^2_{L^2}=\Lambda_2(m(\xi_1)m(\xi_2)).
\end{align}
Then using Proposition \ref{prop-1} and Remark \ref{rem-modi} we find
\begin{align}\label{equ-M3}
\frac{d}{dt}E^2_I(t)=\Lambda_3(M_3), \quad M_3(\xi_1, \xi_2, \xi_3)=-i[m(\xi_1)m(\xi_2+\xi_3)\{\xi_2+\xi_3\}]_{sym}.
\end{align}
Set
\begin{align}\label{equ-energy-3}
E^3_I(t)=E^2_I(t)+\Lambda_3(\sigma_3),\quad \beta_3=-\frac{M_3}{\alpha_3},
\end{align}
then by Proposition \ref{prop-1} and Remark \ref{rem-modi} again, we have
\begin{align}\label{equ-M4}
\frac{d}{dt}E^3_I(t)=\Lambda_4(M_4), \quad M_4(\xi_1, \xi_2, \xi_3,\xi_4)=-i\frac{3}{2}[\beta_3(\xi_1,\xi_2,\xi_3+\xi_4)\{\xi_3+\xi_4\}]_{sym}.
\end{align}
Moreover, defining
\begin{align}\label{equ-energy-4}
E^4_I(t)=E^3_I(t)+\Lambda_4(\sigma_4),\quad \beta_4=-\frac{M_4}{\alpha_4},
\end{align}
we have
\begin{align}\label{equ-beta4-0}
\frac{d}{dt}E^4_I(t)=\Lambda_5(M_5), \quad M_5(\xi_1, \xi_2, \xi_3,\xi_4,\xi_5)=-2i[\beta_4(\xi_1,\xi_2,\xi_3,\xi_4+\xi_5)\{\xi_4+\xi_5\}]_{sym}.
\end{align}

\begin{lemma}\label{lem-M4}
If $m$ is even and $\mathbb{R}-$valued and $M_4$ is given by  \eqref{equ-M4}, then the following identity holds
\begin{equation}\label{equ-M4-0}
\begin{split}
M_4(\xi_1,\xi_2,\xi_3,\xi_4) &= -\frac{c}{108}\frac{\alpha_4}{\xi_1\xi_2\xi_3\xi_4}[m^2(\xi_1)+m^2(\xi_2)+m^2(\xi_3)+m^2(\xi_4)
\\&\qquad \qquad \qquad \quad -m^2(\xi_1+\xi_2)-m^2(\xi_1+\xi_3)-m^2(\xi_1+\xi_4)]\\
&{}\quad+\frac{c}{36}\left\{\frac{m^2(\xi_1)}{\xi_1}+\frac{m^2(\xi_2)}{\xi_2}+\frac{m^2(\xi_3)}{\xi_3}+\frac{m^2(\xi_4)}{\xi_4}\right\},
\end{split}
\end{equation}
where $c$ is an absolute constant, $\alpha_4$ is given by \eqref{equ-alpha-k}.
\end{lemma}

Moreover, it is easy to show that, on the hyperplane $\xi_1+\xi_2+\xi_3+\xi_4=0$,
\begin{align}
\alpha_4&=\xi_1^3+\xi_2^3+\xi_3^3+\xi_4^3\label{equ-alpha4-1}\\
&=3(\xi_1\xi_2\xi_3+\xi_1\xi_2\xi_4+\xi_1\xi_3\xi_4+\xi_2\xi_3\xi_4)\label{equ-alpha4-2}\\
&=3(\xi_1+\xi_2)(\xi_1+\xi_3)(\xi_1+\xi_4)\label{equ-alpha4-3}.
\end{align}

\section{Point-wise bounds for $M_4$ and $\beta_4$ }

The goal of this section is to give some point-wise bounds for multipliers $M_4$ and $\beta_4$, which will play an important role in the proof of almost conversation law. The multiplier $m$ in $M_4$, needed in this paper, is given as follows.

Let $\sigma>0$. Set
\begin{align}\label{equ-m}
m(\xi)=\frac{e^{\sigma \xi}+ e^{-\sigma \xi}}{2}, \quad \xi \in \mathbb{R}.
\end{align}
It is easy to see that
$$
e^{\sigma|\xi|}/2\leq m(\xi)\leq e^{\sigma|\xi|}, \quad \xi \in \mathbb{R},
$$
from which, we find
\begin{align}\label{equ-equivalent}
\|f\|_{G^\sigma}/2\leq \|m(D)f\|_{L^2}\leq \|f\|_{G^\sigma}, \quad f\in \mathscr{S}.
\end{align}
In other words, $\|m(D)\cdot\|_{L^2}$ is an equivalent norm of $\|\cdot\|_{G^\sigma}$.

By Taylor expansion, we have
\begin{align}\label{equ-taylor}
m(\xi)=\sum_{k=0}^{\infty}\frac{(\sigma \xi)^{2k}}{(2k)!}, \quad \xi \in \mathbb{R}.
\end{align}
Using \eqref{equ-taylor}, we deduce from Lemma \ref{lem-M4} that
\begin{equation}\label{equ-M4-1}
\begin{split}
M_4(\xi_1,\xi_2,\xi_3,\xi_4)= &-\frac{c}{108}\frac{\alpha_4}{\xi_1\xi_2\xi_3\xi_4}\sum_{k=0}^{\infty}\frac{\sigma^{2k}}{(2k)!}\Big[\xi_1^{2k}+\xi_2^{2k}+\xi_1^{2k}+\xi_4^{2k}
\\&\qquad \qquad \qquad \quad -(\xi_1+\xi_2)^{2k}-(\xi_1+\xi_3)^{2k}-(\xi_1+\xi_4)^{2k}\Big]\\
&  +\frac{c}{36}\sum_{k=0}^{\infty}\frac{\sigma^{2k}}{(2k)!}\Big(\xi_1^{2k-1}+\xi_2^{2k-1}+\xi_3^{2k-1}+\xi_4^{2k-1}\Big).
\end{split}
\end{equation}
Note that the terms in the sum of \eqref{equ-M4-1} are polynomials of $\xi_1, \xi_2, \xi_3, \xi_4$, which allow us to obtain cancelation conveniently on   the hyperplane
\begin{align}\label{equ-hyper-4}
\xi_1+\xi_2+\xi_3+\xi_4=0.
\end{align}

Before giving a detailed analysis of $M_4$, we first show that the terms on the right hand side of \eqref{equ-M4-1} vanishes on the hyperplane \eqref{equ-hyper-4} if $k=0,1$. In other words, the sums in \eqref{equ-M4-1} are only taking for $k\geq 2$. In fact, in the case $k=0$, using the property \eqref{equ-alpha4-2} of $\alpha_4$, we find
\begin{align*}
RHS \eqref{equ-M4-1}&=-\frac{c}{108}\frac{\alpha_4}{\xi_1\xi_2\xi_3\xi_4} + \frac{c}{36}\Big(\xi_1^{-1}+\xi_2^{-1}+\xi_3^{-1}+\xi_4^{-1}\Big)\\
&=\frac{c}{36}\frac{1}{\xi_1\xi_2\xi_3\xi_4}\Big(\xi_1\xi_2\xi_3+\xi_1\xi_2\xi_4+\xi_1\xi_3\xi_3\xi_4+\xi_2\xi_3\xi_4-\frac{\alpha_4}{3}\Big)=0.
\end{align*}
In the case $k=1$, we have
\begin{align*}
RHS \eqref{equ-M4-1} &= -\frac{c}{108}\frac{\alpha_4}{\xi_1\xi_2\xi_3\xi_4}\frac{\sigma^{2}}{2}\Big[\xi_1^{2}+\xi_2^{2}+\xi_1^{2}+\xi_4^{2} -(\xi_1+\xi_2)^{2}-(\xi_1+\xi_3)^{2}-(\xi_1+\xi_4)^{2}\Big]\\
& \quad +\frac{c}{36}\frac{\sigma^{2}}{2}(\xi_1+\xi_2+\xi_3+\xi_4)\\
&=\frac{\sigma^2}{216}(\xi_1+\xi_2+\xi_3+\xi_4)\Big(3+\frac{\alpha_4}{\xi_2\xi_3\xi_4}\Big)=0.
\end{align*}
Thus, $M_4$ can be rewritten as
\begin{equation}\label{equ-M4-2}
\begin{split}
M_4(\xi_1,\xi_2,\xi_3,\xi_4) = & -\frac{c}{108}\frac{\alpha_4}{\xi_1\xi_2\xi_3\xi_4}\sum_{k=2}^{\infty}\frac{\sigma^{2k}}{(2k)!}\Omega_1(k;\xi_1,\xi_2,\xi_3,\xi_4)\\
&+\frac{c}{36}\sum_{k=1}^{\infty}\frac{\sigma^{2(k+1)}}{(2(k+1))!}\Omega_2(k;\xi_1,\xi_2,\xi_3,\xi_4),
\end{split}
\end{equation}
where $\Omega_1$ and $\Omega_2$ are given by
\begin{align}\label{equ-omega-1}
\Omega_1(k;\xi_1,\xi_2,\xi_3,\xi_4) :=  \xi_1^{2k}+\xi_2^{2k}+\xi_1^{2k}+\xi_4^{2k}-(\xi_1+\xi_2)^{2k}-(\xi_1+\xi_3)^{2k}-(\xi_1+\xi_4)^{2k},
\end{align}
\begin{align}\label{equ-omega-2}
\Omega_2(k;\xi_1,\xi_2,\xi_3,\xi_4) :=\xi_1^{2k+1}+\xi_2^{2k+1}+\xi_3^{2k+1}+\xi_4^{2k+1}.
\end{align}
In order to obtain bounds of $M_4$ and $\beta_4$, according to \eqref{equ-M4-2} and \eqref{equ-energy-4}, we need to control
$$
\frac{\Omega_1}{\xi_1\xi_2\xi_3\xi_4}, \quad \frac{\Omega_2}{\alpha_4}.
$$
At a first glance, there are singularities in the two terms. But this is not the case on the hyperplane $\xi_1+\xi_2+\xi_3+\xi_4=0$. We report the fact in the following two subsections.

%
%

\subsection{Decomposition 1}\label{subsec-3.1}
In this subsection, we get rid of the singularity of $\frac{\Omega_1}{\xi_1\xi_2\xi_3\xi_4}$. To this end, we shall show that for $k\geq 2$
$$
\Omega_1=\xi_1^{2k}+\xi_2^{2k}+\xi_1^{2k}+\xi_4^{2k}-(\xi_1+\xi_2)^{2k}-(\xi_1+\xi_3)^{2k}-(\xi_1+\xi_4)^{2k} = \xi_1\xi_2\xi_3\xi_4 \cdot \textmd{ a polynomial}
$$
on the  hyperplane $\xi_1+\xi_2+\xi_3+\xi_4=0$. This is contained in the following lemma, in which we give a formula of the polynomial.

\begin{lemma}\label{lem-decomp-1}
Assume that $\mathbb{N}\ni k\geq 2$ and $\Omega_1$ is given by \eqref{equ-omega-1}. Then
\begin{equation}\label{equ-decomp1-0}
\begin{split}
\frac{\Omega_1(k)}{\xi_1\xi_2\xi_3\xi_4} = &\sum_{i+j=2k-5}(-1)^i \bigg((\xi_1^{j+1}+\xi_2^{j+1})\sum_{m+l=i}\xi_3^m(\xi_3+\xi_4)^l +(\xi_1^{i+1}-\xi_3^{i+1}) \sum_{m+l=j}\xi_2^m(\xi_2+\xi_4)^l\\
& + (\xi_2^{i+1}-\xi_3^{i+1})\sum_{m+l=j}\xi_1^m(\xi_1+\xi_4)^l +(\xi_1^{i+1}+\xi_2^{i+1})\sum_{\substack{m+l=j}}\xi_4^{m}(\xi_3+\xi_4)^{l}\bigg)\\
&-2 \sum_{i+j=2k-4}\bigg(\xi_1^i(\xi_1+\xi_4)^j+\xi_2^i(\xi_2+\xi_4)^j+ \xi_3^i(\xi_3+\xi_4)^j +\xi_4^{i}(\xi_3+\xi_4)^j\bigg)\\
&  -\sum_{i+j=2k-4}(-1)^i\bigg(\xi_4^{i}  \sum_{\substack{m+l=j}} \xi_1^{m}(\xi_1+\xi_3)^l +\xi_3^{i}  \sum_{\substack{m+l=j}} \xi_4^{m}(\xi_2+\xi_4)^l\bigg)\\
& -\sum_{i+j=2k-4}(-1)^i\bigg(\xi_4^{i}  \sum_{\substack{m+l=j}} (\xi_2^{m}+\xi_3^{m})(\xi_2+\xi_3)^l \bigg).
\end{split}
\end{equation}
\end{lemma}

\begin{remark}\label{rem-lem1-1}
The sums in \eqref{equ-decomp1-0} are taking for all nonnegative numbers. For example,
$$
\sum_{m+l=i}\cdots = \sum_{m+l=i; m,l\geq 0}\cdots.
$$
Moreover, the sum vanishes if the sum taking over the empty set. For example,
$$
\sum_{i+j=-1}\cdots = \sum_{i+j=-1; i,j\geq 0}\cdots = 0.
$$
The sums in the rest of the paper are understood in the same way.
\end{remark}
\begin{remark}\label{rem-lem1-2}
In particular, setting $k=2$ in \eqref{equ-decomp1-0}, using Remark \ref{rem-lem1-1}, we find
\begin{align}\label{equ-rem-lem1-2}
\Omega_1(2; \xi_1,\xi_2,\xi_3,\xi_4)=-12\xi_1\xi_2\xi_3\xi_4.
\end{align}
\end{remark}

\emph{\textbf{Proof of Lemma \ref{lem-decomp-1}.}} The proof is long and the computation is complicated. But the reader can build some intuitions by working out \eqref{equ-rem-lem1-2} following our strategy in the sequel.

We divide the discussion into four steps.

\textbf{Step 1. Find $\frac{\Omega_1}{\xi_1}$.} Rewrite $\Omega_1$ as
$$
\Omega_1 = \xi_1^{2k}+\xi_2^{2k}-(\xi_1+\xi_2)^{2k}+\xi_3^{2k}-(\xi_1+\xi_3)^{2k}+\xi_4^{2k}-(\xi_1+\xi_4)^{2k}.
$$
Using the elementary indentity $x^n-y^n=(x-y)\sum_{i+j=n-1}x^iy^j$\footnote{We shall use the fact implicity in the sequel.} we find
\begin{align*}
\Omega_1 = \xi_1^{2k}-\xi_1 \sum_{i+j=2k-1}\xi_2^i(\xi_1+\xi_2)^j+\xi_3^i(\xi_1+\xi_3)^j+\xi_4^i(\xi_1+\xi_4)^j.
\end{align*}
Thus, we have
\begin{align}\label{equ-decomp1-1}
\frac{\Omega_1}{\xi_1} = \xi_1^{2k-1}-\sum_{i+j=2k-1}\xi_2^i(\xi_1+\xi_2)^j+\xi_3^i(\xi_1+\xi_3)^j+\xi_4^i(\xi_1+\xi_4)^j.
\end{align}

\textbf{Step 2. Find $\frac{\Omega_1}{\xi_1\xi_2}$.} Split
\begin{align}\label{equ-decomp1-2}
\sum_{i+j=2k-1}\xi_2^i(\xi_1+\xi_2)^j = (\xi_1+\xi_2)^{2k-1}+ \sum_{i+j=2k-1, i\geq 1}\xi_2^i(\xi_1+\xi_2)^j,
\end{align}
and use $\xi_1+\xi_2+\xi_3+\xi_4=0$ to write
\begin{align}\label{equ-decomp1-3}
\xi_3^i(\xi_1+\xi_3)^j = \xi_3^i(-\xi_2-\xi_4)^j, \quad \xi_4^i(\xi_1+\xi_4)^j = \xi_4^i(-\xi_2-\xi_3)^j.
\end{align}
Inserting \eqref{equ-decomp1-2} and \eqref{equ-decomp1-3} into \eqref{equ-decomp1-1} we obtain
\begin{multline}\label{equ-decomp1-4}
\frac{\Omega_1}{\xi_1} = \xi_1^{2k-1}-(\xi_1+\xi_2)^{2k-1}-  \sum_{i+j=2k-1, i\geq 1}\xi_2^i(\xi_1+\xi_2)^j\\
-\sum_{i+j=2k-1}\Big(\xi_3^i(-\xi_2-\xi_4)^j+\xi_4^i(-\xi_2-\xi_3)^j\Big).
\end{multline}
On one hand, we have
\begin{multline}\label{equ-decomp1-4.5}
\xi_1^{2k-1}-(\xi_1+\xi_2)^{2k-1}-  \sum_{i+j=2k-1, i\geq 1}\xi_2^i(\xi_1+\xi_2)^j\\
=-\xi_2\sum_{i+j=2k-2}\xi_1^i(\xi_1+\xi_2)^{j}-  \xi_2\sum_{i+j=2k-1, i\geq 1}\xi_2^{i-1}(\xi_1+\xi_2)^j.
\end{multline}
Since $\sum_{i+j=2k-1, i\geq 1}\xi_2^{i-1}(\xi_1+\xi_2)^j = \sum_{i+j=2k-2}\xi_2^i(\xi_1+\xi_2)^{j}$,
\begin{align}\label{equ-decomp1-4.6}
RHS \eqref{equ-decomp1-4.5} = -\xi_2\sum_{i+j=2k-2}(\xi_1^i+\xi_2^i)(\xi_1+\xi_2)^{j}.
\end{align}
On the other hand,
\begin{multline}\label{equ-decomp1-4.7}
-\sum_{i+j=2k-1}\Big(\xi_3^i(-\xi_2-\xi_4)^j+\xi_4^i(-\xi_2-\xi_3)^j\Big) = -\sum_{i+j=2k-1}\Big(\xi_3^i(-\xi_2-\xi_4)^j+\xi_4^j(-\xi_2-\xi_3)^i\Big)\\
=\sum_{i+j=2k-1}(-1)^i\Big(\xi_3^i(\xi_2+\xi_4)^j-\xi_4^j(\xi_2+\xi_3)^i\Big).
\end{multline}
Split the sum in \eqref{equ-decomp1-4.7} into cases: (1) $i,j\geq 1$, (2) $i=0$, (3) $j=0$. Then we obtain
\begin{equation}\label{equ-decomp1-6}
\begin{split}
RHS \eqref{equ-decomp1-4.7} = &\sum_{i+j=2k-1, i,j\geq 1}(-1)^i\bigg(\xi_3^i(\xi_2+\xi_4)^j-\xi_4^j(\xi_2+\xi_3)^i\bigg)\\
&+(\xi_2+\xi_4)^{2k-1}-\xi_4^{2k-1} + (\xi_2+\xi_3)^{2k-1}-\xi_3^{2k-1}.
\end{split}
\end{equation}
Rewrite the term in the sum of \eqref{equ-decomp1-6} as
\begin{align}\label{equ-decomp1-7}
\xi_3^i(\xi_2+\xi_4)^j-\xi_4^j(\xi_2+\xi_3)^i &= \xi_3^i((\xi_2+\xi_4)^j-\xi_4^j)-\xi_4^j((\xi_2+\xi_3)^i-\xi_3^i)\nonumber\\
&=\xi_3^i \xi_2 \sum_{m+l=j-1} \xi_4^m(\xi_2+\xi_4)^l - \xi_4^j \xi_2 \sum_{m+l=i-1} \xi_3^m(\xi_2+\xi_3)^l.
\end{align}
Thanks to \eqref{equ-decomp1-7}, we deduce from \eqref{equ-decomp1-6} that
\begin{equation}\label{equ-decomp1-8}
\begin{split}
RHS \eqref{equ-decomp1-4.7}= &\sum_{i+j=2k-1, i,j\geq 1}(-1)^i\xi_2\bigg(\xi_3^i  \sum_{m+l=j-1} \xi_4^m(\xi_2+\xi_4)^l - \xi_4^j  \sum_{m+l=i-1} \xi_3^m(\xi_2+\xi_3)^l\bigg)\\
&+\xi_2 \bigg(\sum_{i+j=2k-2}\xi_4^i(\xi_2+\xi_4)^{j} + \xi_2\sum_{i+j=2k-2}\xi_3^i(\xi_2+\xi_3)^{j}\bigg).
\end{split}
\end{equation}
Combining \eqref{equ-decomp1-4.6} and \eqref{equ-decomp1-8} that
\begin{equation}\label{equ-decomp1-9}
\begin{split}
\frac{\Omega_1}{\xi_1\xi_2}= &-\sum_{i+j=2k-2}(\xi_1^i+\xi_2^i)(\xi_1+\xi_2)^{j}+\sum_{i+j=2k-2}\xi_4^i(\xi_2+\xi_4)^{j} + \sum_{i+j=2k-2}\xi_3^i(\xi_2+\xi_3)^{j}\\
&+\sum_{i+j=2k-1, i,j\geq 1}(-1)^i\bigg(\xi_3^i  \sum_{m+l=j-1} \xi_4^m(\xi_2+\xi_4)^l - \xi_4^j  \sum_{m+l=i-1} \xi_3^m(\xi_2+\xi_3)^l\bigg).
\end{split}
\end{equation}

\textbf{Step 3. Find $\frac{\Omega_1}{\xi_1\xi_2\xi_3}$.} Using $\xi_1+\xi_2+\xi_3+\xi_4=0$ again, we rewrite \eqref{equ-decomp1-9} as
\begin{align}\label{equ-decomp1-10}
\frac{\Omega_1}{\xi_1\xi_2} &= \sum_{i+j=2k-2}\xi_4^i(-\xi_1-\xi_3)^{j}-\xi_1^i(-\xi_3-\xi_4)^j+\sum_{i+j=2k-2}\xi_3^i(\xi_2+\xi_3)^{j} -\sum_{i+j=2k-2}\xi_2^i(-\xi_3-\xi_4)^{j}\nonumber\\
&\quad \, +\sum_{i+j=2k-1, i,j\geq 1}(-1)^i\bigg(\xi_3^i  \sum_{m+l=j-1} \xi_4^m(\xi_2+\xi_4)^l - \xi_4^j  \sum_{m+l=i-1} \xi_3^m(\xi_2+\xi_3)^l\bigg)\nonumber\\
&= A_1 +A_2+A_3,
\end{align}
where $A_1, A_2, A_3$ are given by
\begin{eqnarray}
A_1 &=& \sum_{i+j=2k-2}\xi_4^i(-\xi_1-\xi_3)^{j}-\xi_1^i(-\xi_3-\xi_4)^j,\label{equ-decomp1-11}\\
A_2 &=&  \sum_{\substack{i+j=2k-1,\\ i,j\geq 1}}(-1)^i\xi_3^i  \sum_{m+l=j-1} \xi_4^m(\xi_2+\xi_4)^l -\sum_{\substack{i+j=2k-1,\\ i,j\geq 1}}(-1)^i \xi_4^j  \sum_{\substack{m+l=i-1,\\ m\geq 1}} \xi_3^m(\xi_2+\xi_3)^l\nonumber\\
& &+  \sum_{\substack{ i+j=2k-2,\\ i\geq 1}}\xi_3^i(\xi_2+\xi_3)^{j},\label{equ-decomp1-11.5}\\
A_3 &=& -\sum_{\substack{i+j=2k-1,\\ i,j\geq 1}}(-1)^i \xi_4^j  \sum_{\substack{m+l=i-1,\\ m=0}} \xi_3^m(\xi_2+\xi_3)^l + (\xi_2+\xi_3)^{2k-2} - \sum_{i+j=2k-2}\xi_2^i(-\xi_3-\xi_4)^{j}\nonumber.
\end{eqnarray}
We simplify $A_3$ as
\begin{align}\label{equ-decomp1-13}
A_3 &= -\sum_{i+j=2k-1, i,j\geq 1}(-1)^i \xi_4^j   (\xi_2+\xi_3)^{i-1} + (\xi_2+\xi_3)^{2k-2} - \sum_{i+j=2k-2}\xi_2^i(-\xi_3-\xi_4)^{j}\nonumber\\
&=\sum_{i+j=2k-1, i,j\geq 1} (-\xi_4)^j   (\xi_2+\xi_3)^{i-1} + (\xi_2+\xi_3)^{2k-2} - \sum_{i+j=2k-2}\xi_2^i(-\xi_3-\xi_4)^{j}\nonumber\\
&=\sum_{i+j=2k-2}(-\xi_4)^j(\xi_2+\xi_3)^i-\sum_{i+j=2k-2}\xi_2^i(-\xi_3-\xi_4)^{j}\nonumber\\
&=\sum_{i+j=2k-2}(-\xi_4)^i(\xi_2+\xi_3)^j-\sum_{i+j=2k-2}\xi_2^j(-\xi_3-\xi_4)^{i}\nonumber\\
&=\sum_{i+j=2k-2}(-1)^i\Big(\xi_4^i(\xi_2+\xi_3)^j-\xi_2^j(\xi_3+\xi_4)^{i}\Big).
\end{align}
To proceed, we rewrite \eqref{equ-decomp1-13} as
\begin{align}\label{equ-decomp1-14}
RHS \eqref{equ-decomp1-13}&=\sum_{i+j=2k-2, i,j\geq 1}(-1)^i\Big(\xi_4^i(\xi_2+\xi_3)^j-\xi_2^j(\xi_3+\xi_4)^{i}\Big)\nonumber\\
& \quad + \xi_4^{2k-2}-(\xi_3+\xi_4)^{2k-2}+(\xi_2+\xi_3)^{2k-2}-\xi_2^{2k-2}\nonumber\\
&=\sum_{i+j=2k-2, i,j\geq 1}(-1)^i\bigg(\xi_4^i\Big((\xi_2+\xi_3)^j-\xi_2^j\Big)-\xi_2^j\Big((\xi_3+\xi_4)^{i}-\xi_4^i\Big)\bigg)\nonumber\\
& \quad + \xi_4^{2k-2}-(\xi_3+\xi_4)^{2k-2}+(\xi_2+\xi_3)^{2k-2}-\xi_2^{2k-2}\nonumber\\
&=\sum_{i+j=2k-2, i,j\geq 1}(-1)^i\xi_3\bigg(\xi_4^i\sum_{m+l=j-1}\xi_2^m(\xi_2+\xi_3)^l-\xi_2^j\sum_{m+l=i-1}\xi_4^m(\xi_3+\xi_4)^{l}\bigg)\nonumber\\
& \quad + \xi_3\sum_{i+j=2k-3}\Big(\xi_2^i(\xi_2+\xi_3)^j-\xi_4^i(\xi_3+\xi_4)^j\Big)\nonumber\\
&=\sum_{i+j=2k-4}(-1)^i\xi_3\bigg(\xi_2^{j+1}\sum_{m+l=i}\xi_4^m(\xi_3+\xi_4)^{l}-\xi_4^{i+1}\sum_{m+l=j}\xi_2^m(\xi_2+\xi_3)^l\bigg)\nonumber\\
& \quad + \xi_3\sum_{i+j=2k-3}\xi_2^i(\xi_2+\xi_3)^j-\xi_4^i(\xi_3+\xi_4)^j.
\end{align}
Replacing $\xi_2$ by $\xi_1$ in \eqref{equ-decomp1-14}, we obtain
\begin{align}\label{equ-decomp1-15}
A_1 &=\sum_{i+j=2k-2}\xi_4^i(-\xi_1-\xi_3)^{j}-\xi_1^i(-\xi_3-\xi_4)^j\nonumber\\
&=\sum_{i+j=2k-2}(-1)^i\Big(\xi_4^i(\xi_1+\xi_3)^j-\xi_1^j(\xi_3+\xi_4)^{i}\Big)\nonumber\\
&=\sum_{i+j=2k-4}(-1)^i\xi_3\bigg(\xi_1^{j+1}\sum_{m+l=i}\xi_4^m(\xi_3+\xi_4)^{l}-\xi_4^{i+1}\sum_{m+l=j}\xi_1^m(\xi_1+\xi_3)^l\bigg)\nonumber\\
& \quad + \xi_3\sum_{i+j=2k-3}\xi_1^i(\xi_1+\xi_3)^j-\xi_4^i(\xi_3+\xi_4)^j.
\end{align}
For $A_2$, we have
\begin{align}\label{equ-decomp1-16}
A_2 &= \xi_3 \sum_{\substack{i+j=2k-1,\\ i,j\geq 1}}(-1)^i\xi_3^{i-1}  \sum_{m+l=j-1} \xi_4^m(\xi_2+\xi_4)^l -\xi_3\sum_{\substack{i+j=2k-1,\\ i,j\geq 1}}(-1)^i \xi_4^j  \sum_{\substack{m+l=i-1,\\ m\geq 1}} \xi_3^{m-1}(\xi_2+\xi_3)^l\nonumber\\
& \quad +  \xi_3\sum_{\substack{ i+j=2k-2,\\ i\geq 1}}\xi_3^{i-1}(\xi_2+\xi_3)^{j}\nonumber\\
&= -\xi_3 \sum_{i+j=2k-3}(-1)^i\xi_3^{i}  \sum_{m+l=j} \xi_4^m(\xi_2+\xi_4)^l +\xi_3\sum_{i+j=2k-3}(-1)^i \xi_4^{j+1}  \sum_{\substack{m+l=i,\\ m\geq 1}} \xi_3^{m-1}(\xi_2+\xi_3)^l\nonumber\\
& \quad +  \xi_3\sum_{i+j=2k-3}\xi_3^{i}(\xi_2+\xi_3)^{j}
\end{align}
It follows from \eqref{equ-decomp1-10}-\eqref{equ-decomp1-16} that
\begin{equation}\label{equ-decomp1-17}
\begin{split}
\frac{\Omega_1}{\xi_1\xi_2\xi_3}= &\sum_{i+j=2k-4}(-1)^i\bigg(\xi_2^{j+1}\sum_{m+l=i}\xi_4^m(\xi_3+\xi_4)^{l}-\xi_4^{i+1}\sum_{m+l=j}\xi_2^m(\xi_2+\xi_3)^l\bigg)\\
&  + \sum_{i+j=2k-3}\xi_2^i(\xi_2+\xi_3)^j-\xi_4^i(\xi_3+\xi_4)^j\\
&+\sum_{i+j=2k-4}(-1)^i\bigg(\xi_1^{j+1}\sum_{m+l=i}\xi_4^m(\xi_3+\xi_4)^{l}-\xi_4^{i+1}\sum_{m+l=j}\xi_1^m(\xi_1+\xi_3)^l\bigg)\\
& + \sum_{i+j=2k-3}\xi_1^i(\xi_1+\xi_3)^j-\xi_4^i(\xi_3+\xi_4)^j\\
& - \sum_{i+j=2k-3}(-1)^i\xi_3^{i}  \sum_{m+l=j} \xi_4^m(\xi_2+\xi_4)^l +\sum_{i+j=2k-3}(-1)^i \xi_4^{j+1}  \sum_{\substack{m+l=i,\\ m\geq 1}} \xi_3^{m-1}(\xi_2+\xi_3)^l\\
&  +  \sum_{i+j=2k-3}\xi_3^{i}(\xi_2+\xi_3)^{j}.
\end{split}
\end{equation}

\textbf{Step 4. Find $\frac{\Omega_1}{\xi_1\xi_2\xi_3\xi_4}$.} For our purpose, we rewrite \eqref{equ-decomp1-17} as
\begin{align}\label{equ-decomp1-18}
\frac{\Omega_1}{\xi_1\xi_2\xi_3} = B_1 + B_2 ,
\end{align}
where $B_1$ consisting of all terms with a explicit factor $\xi_4$, $B_2$ consisting of the remainder terms.  More precisely, $B_1, B_2$ are given by
\begin{align}\label{equ-decomp1-19}
B_1= &\sum_{i+j=2k-4}(-1)^i\bigg(\xi_2^{j+1}\sum_{\substack{m+l=i,\\ m\geq 1}}\xi_4^m(\xi_3+\xi_4)^{l}-\xi_4^{i+1}\sum_{m+l=j}\xi_2^m(\xi_2+\xi_3)^l\bigg) + \sum_{\substack{i+j=2k-3,\\ i\geq 1}}-\xi_4^i(\xi_3+\xi_4)^j\nonumber\\
& +\sum_{i+j=2k-4}(-1)^i\bigg(\xi_1^{j+1}\sum_{\substack{m+l=i,\\ m\geq 1}}\xi_4^m(\xi_3+\xi_4)^{l}-\xi_4^{i+1}\sum_{m+l=j}\xi_1^m(\xi_1+\xi_3)^l\bigg) + \sum_{\substack{i+j=2k-3, \\i\geq 1}}-\xi_4^i(\xi_3+\xi_4)^j\nonumber\\
&  - \sum_{i+j=2k-3}(-1)^i\xi_3^{i}  \sum_{\substack{m+l=j,\\ m\geq 1}} \xi_4^m(\xi_2+\xi_4)^l +\sum_{i+j=2k-3}(-1)^i \xi_4^{j+1}  \sum_{\substack{m+l=i,\\ m\geq 1}} \xi_3^{m-1}(\xi_2+\xi_3)^l,
\end{align}
\begin{align}\label{equ-decomp1-20}
B_2= &\sum_{i+j=2k-4}(-1)^i\xi_2^{j+1}\sum_{\substack{m+l=i,\\ m=0}}\xi_4^m(\xi_3+\xi_4)^{l}+\sum_{i+j=2k-3}\xi_2^i(\xi_2+\xi_3)^j+ \sum_{\substack{i+j=2k-3,\\ i=0}}-\xi_4^i(\xi_3+\xi_4)^j\nonumber\\
&\quad +\sum_{i+j=2k-4}(-1)^i\xi_1^{j+1}\sum_{\substack{m+l=i,\\ m=0}}\xi_4^m(\xi_3+\xi_4)^{l} + \sum_{i+j=2k-3}\xi_1^i(\xi_1+\xi_3)^j+\sum_{\substack{i+j=2k-3, \\i=0}}-\xi_4^i(\xi_3+\xi_4)^j\nonumber\\
& \quad - \sum_{i+j=2k-3}(-1)^i\xi_3^{i}  \sum_{\substack{m+l=j,\\ m=0}} \xi_4^m(\xi_2+\xi_4)^l  +  \sum_{i+j=2k-3}\xi_3^{i}(\xi_2+\xi_3)^{j}.
\end{align}
\textbf{Contribution of $B_1$.} It follows from \eqref{equ-decomp1-19} that
\begin{align}\label{equ-decomp1-201}
\frac{B_1}{\xi_4} &= \sum_{i+j=2k-4}(-1)^i\bigg(\xi_2^{j+1}\sum_{\substack{m+l=i,\\ m\geq 1}}\xi_4^{m-1}(\xi_3+\xi_4)^{l}-\xi_4^{i}\sum_{m+l=j}\xi_2^m(\xi_2+\xi_3)^l\bigg) + \sum_{\substack{i+j=2k-3,\\ i\geq 1}}-\xi_4^{i-1}(\xi_3+\xi_4)^j\nonumber\\
& \quad +\sum_{i+j=2k-4}(-1)^i\bigg(\xi_1^{j+1}\sum_{\substack{m+l=i,\\ m\geq 1}}\xi_4^{m-1}(\xi_3+\xi_4)^{l}-\xi_4^{i}\sum_{m+l=j}\xi_1^m(\xi_1+\xi_3)^l\bigg) + \sum_{\substack{i+j=2k-3, \\i\geq 1}}-\xi_4^{i-1}(\xi_3+\xi_4)^j\nonumber\\
& \quad - \sum_{i+j=2k-3}(-1)^i\xi_3^{i}  \sum_{\substack{m+l=j,\\ m\geq 1}} \xi_4^{m-1}(\xi_2+\xi_4)^l +\sum_{i+j=2k-3}(-1)^i \xi_4^{j}  \sum_{\substack{m+l=i,\\ m\geq 1}} \xi_3^{m-1}(\xi_2+\xi_3)^l\nonumber\\
&= \sum_{i+j=2k-4}(-1)^i\bigg(\xi_2^{j+1}\sum_{\substack{m+l=i-1}}\xi_4^{m}(\xi_3+\xi_4)^{l}-\xi_4^{i}\sum_{m+l=j}\xi_2^m(\xi_2+\xi_3)^l\bigg) - \sum_{i+j=2k-4}\xi_4^{i}(\xi_3+\xi_4)^j\nonumber\\
&\quad +\sum_{i+j=2k-4}(-1)^i\bigg(\xi_1^{j+1}\sum_{\substack{m+l=i-1}}\xi_4^{m}(\xi_3+\xi_4)^{l}-\xi_4^{i}\sum_{m+l=j}\xi_1^m(\xi_1+\xi_3)^l\bigg) - \sum_{\substack{i+j=2k-4}}\xi_4^{i}(\xi_3+\xi_4)^j\nonumber\\
& \quad - \sum_{i+j=2k-3}(-1)^i\xi_3^{i}  \sum_{\substack{m+l=j-1}} \xi_4^{m}(\xi_2+\xi_4)^l +\sum_{i+j=2k-3}(-1)^i \xi_4^{j}  \sum_{\substack{m+l=i-1}} \xi_3^{m}(\xi_2+\xi_3)^l.
\end{align}
Using the fact
\begin{multline*}
 - \sum_{i+j=2k-3}(-1)^i\xi_3^{i}  \sum_{\substack{m+l=j-1}} \xi_4^{m}(\xi_2+\xi_4)^l +\sum_{i+j=2k-3}(-1)^i \xi_4^{j}  \sum_{\substack{m+l=i-1}} \xi_3^{m}(\xi_2+\xi_3)^l\\
= -\sum_{i+j=2k-4}(-1)^i\bigg(\xi_4^{i}  \sum_{\substack{m+l=j}} \xi_3^{m}(\xi_2+\xi_3)^l +\xi_3^{i}  \sum_{\substack{m+l=j}} \xi_4^{m}(\xi_2+\xi_4)^l\bigg)
\end{multline*}
and rearranging the terms in  \eqref{equ-decomp1-201}, we obtain
\begin{equation}\label{equ-decomp1-205}
\begin{split}
\frac{B_1}{\xi_4}= &-2 \sum_{\substack{i+j=2k-4}}\xi_4^{i}(\xi_3+\xi_4)^j+\sum_{i+j=2k-4}(-1)^i\bigg((\xi_1^{i+1}+\xi_2^{i+1})\sum_{\substack{m+l=j-1}}\xi_4^{m}(\xi_3+\xi_4)^{l}\bigg)\\
&  -\sum_{i+j=2k-4}(-1)^i\bigg(\xi_4^{i}  \sum_{\substack{m+l=j}} \xi_1^{m}(\xi_1+\xi_3)^l +\xi_3^{i}  \sum_{\substack{m+l=j}} \xi_4^{m}(\xi_2+\xi_4)^l\bigg)\\
& -\sum_{i+j=2k-4}(-1)^i\bigg(\xi_4^{i}  \sum_{\substack{m+l=j}} (\xi_2^{m}+\xi_3^{m})(\xi_2+\xi_3)^l \bigg).
\end{split}
\end{equation}

\textbf{Contribution of $B_2$.} Using $\xi_1+\xi_2+\xi_3+\xi_4=0$, we rewrite \eqref{equ-decomp1-20} as
\begin{align}\label{equ-decomp1-21}
B_2 &= \sum_{i+j=2k-4}(-1)^i\xi_2^{j+1}(\xi_3+\xi_4)^{i}+\sum_{i+j=2k-3}\xi_2^i(-\xi_1-\xi_4)^j-(\xi_3+\xi_4)^{2k-3}\nonumber\\
&\quad +\sum_{i+j=2k-4}(-1)^i\xi_1^{j+1}(\xi_3+\xi_4)^{i} + \sum_{i+j=2k-3}\xi_1^i(-\xi_2-\xi_4)^j-(\xi_3+\xi_4)^{2k-3}\nonumber\\
& \quad - \sum_{i+j=2k-3}(-1)^i\xi_3^{i} (\xi_2+\xi_4)^j +  \sum_{i+j=2k-3}\xi_3^{i}(-\xi_1-\xi_4)^{j}\nonumber\\
&= \sum_{i+j=2k-4}(-1)^i\xi_2^{j+1}(\xi_3+\xi_4)^{i}- \sum_{i+j=2k-3}(-1)^i\xi_3^{i} (\xi_2+\xi_4)^j-(\xi_3+\xi_4)^{2k-3}\nonumber\\
&\quad +\sum_{i+j=2k-4}(-1)^i\xi_1^{j+1}(\xi_3+\xi_4)^{i} +  \sum_{i+j=2k-3}\xi_3^{i}(-\xi_1-\xi_4)^{j}-(\xi_3+\xi_4)^{2k-3}\nonumber\\
& \quad +\sum_{i+j=2k-3}\xi_2^i(-\xi_1-\xi_4)^j+ \sum_{i+j=2k-3}\xi_1^i(-\xi_2-\xi_4)^j\nonumber\\
&:=B_{21} +B_{22}+B_{23}.
\end{align}
We deal with $B_{21},B_{22},B_{23}$ as follows. For $B_{21}$,
\begin{align}\label{equ-decomp1-22}
B_{21} &= \sum_{i+j=2k-4}(-1)^i\xi_2^{j+1}(\xi_3+\xi_4)^{i}- \sum_{i+j=2k-3}(-1)^i\xi_3^{i} (\xi_2+\xi_4)^j-(\xi_3+\xi_4)^{2k-3}\nonumber\\
&=\sum_{i+j=2k-4}(-1)^i\xi_2^{j+1}(\xi_3+\xi_4)^{i}- \sum_{\substack{i+j=2k-3,\\j \geq 1}}(-1)^i\xi_3^{i} (\xi_2+\xi_4)^j +\xi_3^{2k-3}-(\xi_3+\xi_4)^{2k-3}\nonumber\\
&=\sum_{i+j=2k-4}(-1)^i\Big(\xi_2^{j+1}(\xi_3+\xi_4)^{i} - \xi_3^{i} (\xi_2+\xi_4)^{j+1}\Big) +\xi_3^{2k-3}-(\xi_3+\xi_4)^{2k-3}\nonumber\\
&=\sum_{i+j=2k-5}(-1)^i\Big(\xi_2^{j+1}(\xi_3+\xi_4)^{i+1} - \xi_3^{i+1} (\xi_2+\xi_4)^{j+1}\Big)\nonumber\\
&\quad +\xi_2^{2k-3}-(\xi_2+\xi_4)^{2k-3} +\xi_3^{2k-3}-(\xi_3+\xi_4)^{2k-3}\nonumber\\
& = \xi_4\sum_{i+j=2k-5}(-1)^i \bigg(\xi_2^{j+1}\sum_{m+l=i}\xi_3^m(\xi_3+\xi_4)^l -\xi_3^{i+1} \sum_{m+l=j}\xi_2^m(\xi_2+\xi_4)^l \bigg)\nonumber\\
&\quad -\xi_4 \sum_{i+j=2k-4}\bigg(\xi_2^i(\xi_2+\xi_4)^j+ \xi_3^i(\xi_3+\xi_4)^j \bigg).
\end{align}
Similarly, we have
\begin{equation}\label{equ-decomp1-23}
\begin{split}
B_{22} = & \quad \xi_4\sum_{i+j=2k-5}(-1)^i \bigg(\xi_1^{j+1}\sum_{m+l=i}\xi_3^m(\xi_3+\xi_4)^l -\xi_3^{i+1} \sum_{m+l=j}\xi_1^m(\xi_1+\xi_4)^l \bigg)\\
& -\xi_4 \sum_{i+j=2k-4}\bigg(\xi_1^i(\xi_1+\xi_4)^j+ \xi_3^i(\xi_3+\xi_4)^j \bigg).
\end{split}
\end{equation}
For $B_{23}$, we have
\begin{align}\label{equ-decomp1-24}
B_{23}&= \sum_{i+j=2k-3}\xi_2^i(-\xi_1-\xi_4)^j+ \sum_{i+j=2k-3}\xi_1^i(-\xi_2-\xi_4)^j\nonumber\\
&=\sum_{i+j=2k-3}\xi_2^j(-\xi_1-\xi_4)^i+ \sum_{i+j=2k-3}\xi_1^i(-\xi_2-\xi_4)^j\nonumber\\
&=\sum_{i+j=2k-3}(-1)^i\bigg(\xi_2^j(\xi_1+\xi_4)^i- \xi_1^i(\xi_2+\xi_4)^j\bigg)\nonumber\\
&=\sum_{\substack{i+j=2k-3,\\ i,j\geq 1}}(-1)^i\bigg(\xi_2^j(\xi_1+\xi_4)^i- \xi_1^i(\xi_2+\xi_4)^j\bigg)+ \xi_1^{2k-3}-(\xi_1+\xi_4)^{2k-3}+\xi_2^{2k-3}-(\xi_2+\xi_4)^{2k-3}\nonumber\\
&=\sum_{\substack{i+j=2k-3,\\ i,j\geq 1}}(-1)^i\xi_4\bigg(\xi_2^j\sum_{m+l=i-1}\xi_1^m(\xi_1+\xi_4)^l- \xi_1^i\sum_{m+l=j-1}\xi_2^m(\xi_2+\xi_4)^l\bigg)\nonumber\\
&\quad -\xi_4 \sum_{i+j=2k-4} \Big(\xi_1^i(\xi_1+\xi_4)^j+\xi_2^i(\xi_2+\xi_4)^j\Big)\nonumber\\
&=\sum_{\substack{i+j=2k-5}}(-1)^i\xi_4\bigg(\xi_1^{i+1}\sum_{m+l=j}\xi_2^m(\xi_2+\xi_4)^l-\xi_2^{j+1}\sum_{m+l=i}\xi_1^m(\xi_1+\xi_4)^l\bigg)\nonumber\\
&\quad -\xi_4 \sum_{i+j=2k-4} \Big( \xi_1^i(\xi_1+\xi_4)^j+\xi_2^i(\xi_2+\xi_4)^j \Big)\nonumber\\
&=\sum_{\substack{i+j=2k-5}}(-1)^i\xi_4\bigg(\xi_1^{i+1}\sum_{m+l=j}\xi_2^m(\xi_2+\xi_4)^l+\xi_2^{i+1}\sum_{m+l=j}\xi_1^m(\xi_1+\xi_4)^l\bigg)\nonumber\\
&\quad -\xi_4 \sum_{i+j=2k-4} \Big(\xi_1^i(\xi_1+\xi_4)^j+\xi_2^i(\xi_2+\xi_4)^j\Big).
\end{align}
It follows from \eqref{equ-decomp1-21}-\eqref{equ-decomp1-24} that
\begin{equation}\label{equ-decomp1-25}
\begin{split}
\frac{B_{2}}{\xi_4} = & \sum_{i+j=2k-5}(-1)^i \bigg((\xi_1^{j+1}+\xi_2^{j+1})\sum_{m+l=i}\xi_3^m(\xi_3+\xi_4)^l +(\xi_1^{i+1}-\xi_3^{i+1}) \sum_{m+l=j}\xi_2^m(\xi_2+\xi_4)^l\\
& \qquad \qquad \qquad \quad+ (\xi_2^{i+1}-\xi_3^{i+1})\sum_{m+l=j}\xi_1^m(\xi_1+\xi_4)^l\bigg)\\
& -2 \sum_{i+j=2k-4}\bigg(\xi_1^i(\xi_1+\xi_4)^j+\xi_2^i(\xi_2+\xi_4)^j+ \xi_3^i(\xi_3+\xi_4)^j \bigg).
\end{split}
\end{equation}
Combining \eqref{equ-decomp1-205} and \eqref{equ-decomp1-25} gives the lemma \ref{lem-decomp-1}.

\subsection{Decomposition 2}\label{subsec-3.2}
In this subsection, we get rid of the singularity of
$$
\frac{\Omega_2}{\alpha_4} = \frac{\Omega_2}{3(\xi_1+\xi_2)(\xi_1+\xi_3)(\xi_1+\xi_4)},
$$
where we used \eqref{equ-alpha4-3}. To this end, we shall show that for $k\geq 2$
$$
\Omega_2=\xi_1^{2k+1}+\xi_2^{2k+1}+\xi_3^{2k+1}+\xi_4^{2k+1} = (\xi_1+\xi_2)(\xi_1+\xi_3)(\xi_1+\xi_4) \cdot \textmd{ a polynomial}
$$
on the  hyperplane $\xi_1+\xi_2+\xi_3+\xi_4=0$. This is contained in the following lemma, in which we give a formula of the polynomial.
\begin{lemma}\label{lem-decomp-2}
Assume that $\mathbb{N}\ni k\geq 1$ and $\Omega_2$ is given by \eqref{equ-omega-2}. Then
\begin{align}\label{equ-decomp2-0}
\frac{\Omega_2(k)}{(\xi_1+\xi_2)(\xi_1+\xi_3)(\xi_1+\xi_4)}= &\sum_{i+j=2k-2}(-1)^i(2\xi_1^i\xi_4^j+\xi_2^i\xi_3^j)\nonumber\\
&+ \sum_{\substack{i+j=2k-1, \\ i,j\geq 1}}\bigg((-\xi_3)^i\sum_{m+l=j-1}\xi_1^m(-\xi_4)^l + \xi_4^j\sum_{m+l=i-1}(-\xi_2)^m\xi_3^l\bigg)\nonumber\\
&+\sum_{\substack{i+j=2k-2, \\ j\geq 1}}(-1)^{i}\sum_{n+h=i}\xi_1^n(-\xi_4)^h\sum_{m+l=j}\xi_2^m(-\xi_4)^l \nonumber\\
& +\sum_{\substack{i+j=2k-2, \\ j\geq 1}}(-1)^{i+1}\xi_4^{i+1}\bigg[\sum_{m+l=j-1}\Big(\xi_3^m(-\xi_2)^l+\xi_4^m(-\xi_1)^l\Big)\nonumber\\
& +\sum_{\substack{m+l=j,\\ m,l\geq 1}}\Big((-\xi_1)^l \sum_{n+h=m-1}\xi_3^n(-\xi_2)^h + (-\xi_2)^m\sum_{n+h=l-1}\xi_4^n(-\xi_1)^h\Big) \bigg].
\end{align}

\end{lemma}

\emph{\textbf{Proof of Lemma \ref{lem-decomp-2}.}}
We divide the analysis into three steps.

\textbf{Step 1. Find $\frac{\Omega_2}{\xi_1+\xi_2}$.} Using $\xi_1+\xi_2+\xi_3+\xi_4=0$, we have
\begin{align}\label{equ-decomp2-1}
\Omega_2 &= \xi_1^{2k+1}+\xi_2^{2k+1}+\xi_3^{2k+1}+\xi_4^{2k+1}\nonumber\\
&=(\xi_1+\xi_2)\sum_{i+j=2k}(-1)^i\xi_1^i\xi_2^j+(\xi_3+\xi_4)\sum_{i+j=2k}(-1)^i\xi_3^i\xi_4^j\nonumber\\
&=(\xi_1+\xi_2)\sum_{i+j=2k}(-1)^i(\xi_1^i\xi_2^j-\xi_3^i\xi_4^j).
\end{align}
From \eqref{equ-decomp2-1}, we obtain
\begin{align}\label{equ-decomp2-2}
\frac{\Omega_2}{\xi_1+\xi_2} =\sum_{i+j=2k}(-1)^i(\xi_1^i\xi_2^j-\xi_3^i\xi_4^j).
\end{align}

\textbf{Step 2. Find $\frac{\Omega_2}{(\xi_1+\xi_2)(\xi_1+\xi_3)}$.} Rewrite \eqref{equ-decomp2-2} as
\begin{align}\label{equ-decomp2-3}
\frac{\Omega_2}{\xi_1+\xi_2} =\sum_{\substack{i+j=2k,\\ i,j\geq 1}}(-1)^i(\xi_1^i\xi_2^j-\xi_3^i\xi_4^j)+\xi_1^{2k}-\xi_3^{2k}+\xi_2^{2k}-\xi_4^{2k}.
\end{align}
On one hand, we have
\begin{align}\label{equ-decomp2-4}
\xi_1^{2k}-\xi_3^{2k}+\xi_2^{2k}-\xi_4^{2k}&=(\xi_1+\xi_3)\sum_{i+j=2k-1}(-1)^i\xi_3^i\xi_1^j +(\xi_2+\xi_4)\sum_{i+j=2k-1}(-1)^i\xi_4^i\xi_2^j\nonumber\\
&=(\xi_1+\xi_3)\sum_{i+j=2k-1}(-1)^i(\xi_3^i\xi_1^j- \xi_4^i\xi_2^j).
\end{align}
On the other hand, we have
\begin{align}\label{equ-decomp2-5}
\sum_{\substack{i+j=2k,\\ i,j\geq 1}}&(-1)^i(\xi_1^i\xi_2^j-\xi_3^i\xi_4^j)= \sum_{\substack{i+j=2k,\\ i,j\geq 1}}(-1)^i\bigg(\xi_1^i(\xi_2^j-(-\xi_4)^j)-\xi_4^j(\xi_3^i-(-\xi_1)^i)\bigg)\nonumber\\
&=\sum_{\substack{i+j=2k,\\ i,j\geq 1}}(-1)^i\bigg(\xi_1^i(\xi_2+\xi_4)\sum_{m+l=j-1}\xi_2^m(-\xi_4)^l - \xi_4^j(\xi_1+\xi_3)\sum_{m+l=i-1}\xi_3^m(-\xi_1)^l  \bigg)\nonumber\\
&= \sum_{\substack{i+j=2k,\\ i,j\geq 1}}(-1)^{i+1}(\xi_1+\xi_3)\bigg(\xi_1^i\sum_{m+l=j-1}\xi_2^m(-\xi_4)^l+ \xi_4^j\sum_{m+l=i-1}\xi_3^m(-\xi_1)^l  \bigg).
\end{align}
Combining \eqref{equ-decomp2-3}-\eqref{equ-decomp2-5} gives that
\begin{equation}\label{equ-decomp2-6}
\begin{split}
\frac{\Omega_2}{(\xi_1+\xi_2)(\xi_1+\xi_3)}=& \sum_{\substack{i+j=2k,\\ i,j\geq 1}}(-1)^{i+1}\bigg(\xi_1^i\sum_{m+l=j-1}\xi_2^m(-\xi_4)^l+ \xi_4^j\sum_{m+l=i-1}\xi_3^m(-\xi_1)^l  \bigg)\\
& +\sum_{i+j=2k-1}(-1)^i(\xi_3^i\xi_1^j- \xi_4^i\xi_2^j).
\end{split}
\end{equation}
Changing variable $i-1\mapsto i, j-1\mapsto j$, we find
\begin{align}\label{equ-decomp2-7}
\sum_{\substack{i+j=2k,\\ i,j\geq 1}}(-1)^{i+1}\xi_1^i\sum_{m+l=j-1}\xi_2^m(-\xi_4)^l = \sum_{i+j=2k-2}(-1)^{i}\xi_1^{i+1}\sum_{m+l=j}\xi_2^m(-\xi_4)^l.
\end{align}
Similarly,
\begin{align}\label{equ-decomp2-8}
\sum_{\substack{i+j=2k,\\ i,j\geq 1}}(-1)^{i+1}\xi_4^j\sum_{m+l=i-1}\xi_3^m(-\xi_1)^l & = \sum_{i+j=2k-2}(-1)^{i}\xi_4^{j+1}\sum_{m+l=i}\xi_3^m(-\xi_1)^l \nonumber\\
&=\sum_{i+j=2k-2}(-1)^{i}\xi_4^{i+1}\sum_{m+l=j}\xi_3^m(-\xi_1)^l.
\end{align}
Inserting \eqref{equ-decomp2-7}-\eqref{equ-decomp2-8} into \eqref{equ-decomp2-6}, we obtain
\begin{equation}\label{equ-decomp2-9}
\begin{split}
\frac{\Omega_2}{(\xi_1+\xi_2)(\xi_1+\xi_3)}=&\sum_{i+j=2k-2}(-1)^{i}\bigg(\xi_1^{i+1}\sum_{m+l=j}\xi_2^m(-\xi_4)^l+\xi_4^{i+1}\sum_{m+l=j}\xi_3^m(-\xi_1)^l\bigg)\\
&\quad +\sum_{i+j=2k-1}(-1)^i(\xi_3^i\xi_1^j- \xi_4^i\xi_2^j).
\end{split}
\end{equation}

\textbf{Step 3. Find $\frac{\Omega_2}{(\xi_1+\xi_2)(\xi_1+\xi_3)(\xi_1+\xi_4)}$.} It suffices to analyze the two terms on the right hand side of \eqref{equ-decomp2-9}. We claim that
\begin{multline}\label{equ-decomp2-10}
\sum_{i+j=2k-1}(-1)^i(\xi_3^i\xi_1^j- \xi_4^i\xi_2^j)=\quad (\xi_1+\xi_4)\sum_{i+j=2k-2}(-1)^i(\xi_1^i\xi_4^j+\xi_2^i\xi_3^j)\\
+(\xi_1+\xi_4)\sum_{\substack{i+j=2k-1, \\ i,j\geq 1}}\bigg((-\xi_3)^i\sum_{m+l=j-1}\xi_1^m(-\xi_4)^l + \xi_4^j\sum_{m+l=i-1}(-\xi_2)^m\xi_3^l\bigg),
\end{multline}

\begin{multline*}
\lefteqn{ \sum_{i+j=2k-2}(-1)^{i}\bigg(\xi_1^{i+1}\sum_{m+l=j}\xi_2^m(-\xi_4)^l+\xi_4^{i+1}\sum_{m+l=j}\xi_3^m(-\xi_1)^l\bigg)}
\end{multline*}
\begin{equation}\label{equ-decomp2-11}
\begin{split}
= & \,\,(\xi_1+\xi_4)\bigg[\sum_{i+j=2k-2}(-1)^i\xi_1^i\xi_4^j+\sum_{\substack{i+j=2k-2, \\ j\geq 1}}(-1)^{i}\sum_{n+h=i}\xi_1^n(-\xi_4)^h\sum_{m+l=j}\xi_2^m(-\xi_4)^l \bigg]\\
& +(\xi_1+\xi_4)\sum_{\substack{i+j=2k-2, \\ j\geq 1}}(-1)^{i+1}\xi_4^{i+1}\bigg[\sum_{m+l=j-1}\Big(\xi_3^m(-\xi_2)^l+\xi_4^m(-\xi_1)^l\Big)\\
& +\sum_{\substack{m+l=j,\\ m,l\geq 1}}\Big((-\xi_1)^l \sum_{n+h=m-1}\xi_3^n(-\xi_2)^h + (-\xi_2)^m\sum_{n+h=l-1}\xi_4^n(-\xi_1)^h\Big) \bigg].
\end{split}
\end{equation}
To prove \eqref{equ-decomp2-10}, we rewrite
\begin{align}\label{equ-decomp2-12}
\sum_{i+j=2k-1}&(-1)^i(\xi_3^i\xi_1^j- \xi_4^i\xi_2^j)= \sum_{i+j=2k-1}\Big((-1)^i\xi_3^i\xi_1^j- (-1)^i\xi_4^i\xi_2^j\Big)\nonumber\\
& \quad  =   \sum_{i+j=2k-1}\Big((-1)^i\xi_3^i\xi_1^j+ (-1)^j\xi_4^i\xi_2^j\Big) = \sum_{i+j=2k-1}\Big((-\xi_3)^i\xi_1^j+ (- \xi_2)^i\xi_4^j\Big) \nonumber\\
&\quad  = \sum_{\substack{i+j=2k-1, \\ i,j\geq 1}}\Big((-\xi_3)^i\xi_1^j+ (- \xi_2)^i\xi_4^j\Big) + \xi_1^{2k-1}+\xi_4^{2k-1}-(\xi_2^{2k-1}+\xi_3^{2k-1}).
\end{align}
On one hand,
\begin{align}\label{equ-decomp2-13}
\xi_1^{2k-1}+\xi_4^{2k-1}-(\xi_2^{2k-1}+\xi_3^{2k-1})&=(\xi_1+\xi_4)\sum_{i+j=2k-2}(-1)^i\xi_1^i\xi_4^j - (\xi_2+\xi_3)\sum_{i+j=2k-2}(-1)^i\xi_2^i\xi_3^j\nonumber\\
&=(\xi_1+\xi_4)\sum_{i+j=2k-2}(-1)^i(\xi_1^i\xi_4^j+\xi_2^i\xi_3^j).
\end{align}
On the other hand,
\begin{align}\label{equ-decomp2-14}
\sum_{\substack{i+j=2k-1, \\ i,j\geq 1}}&(-\xi_3)^i\xi_1^j+ (- \xi_2)^i\xi_4^j= \sum_{\substack{i+j=2k-1, \\ i,j\geq 1}}(-\xi_3)^i(\xi_1^j-(-\xi_4)^j)+ (-\xi_3)^i(-\xi_4)^j+(- \xi_2)^i\xi_4^j\nonumber\\
&=\sum_{\substack{i+j=2k-1, \\ i,j\geq 1}}(-\xi_3)^i(\xi_1^j-(-\xi_4)^j)+\xi_4^j((-\xi_2)^i-\xi_3^i)\nonumber\\
&=\sum_{\substack{i+j=2k-1, \\ i,j\geq 1}}(-\xi_3)^i(\xi_1+\xi_4)\sum_{m+l=j-1}\xi_1^m(-\xi_4)^l-\xi_4^j(\xi_2+\xi_3)\sum_{m+l=i-1}(-\xi_2)^m\xi_3^l\nonumber\\
&=\sum_{\substack{i+j=2k-1, \\ i,j\geq 1}}(\xi_1+\xi_4)\bigg((-\xi_3)^i\sum_{m+l=j-1}\xi_1^m(-\xi_4)^l + \xi_4^j\sum_{m+l=i-1}(-\xi_2)^m\xi_3^l\bigg).
\end{align}
Then combining \eqref{equ-decomp2-12}-\eqref{equ-decomp2-14} gives \eqref{equ-decomp2-10}.

To prove \eqref{equ-decomp2-11}, we split the sum into two cases: $j=0, j\geq 1$. We deal with each case as follows. At first,
\begin{multline}\label{equ-decomp2-15}
\sum_{\substack{i+j=2k-2, \\ j=0}}(-1)^{i}\bigg(\xi_1^{i+1}\sum_{m+l=j}\xi_2^m(-\xi_4)^l+\xi_4^{i+1}\sum_{m+l=j}\xi_3^m(-\xi_1)^l\bigg)\\
=\xi_1^{2k-1}+\xi_4^{2k-1}=(\xi_1+\xi_4)\sum_{i+j=2k-2}(-1)^i\xi_1^i\xi_4^j.
\end{multline}
Second, for the case $j\geq 1$ we have
\begin{align}\label{equ-decomp2-16}
&\sum_{\substack{i+j=2k-2, \\ j\geq 1}}(-1)^{i}\bigg(\xi_1^{i+1}\sum_{m+l=j}\xi_2^m(-\xi_4)^l+\xi_4^{i+1}\sum_{m+l=j}\xi_3^m(-\xi_1)^l\bigg)\nonumber\\
&\quad =\sum_{\substack{i+j=2k-2, \\ j\geq 1}}(-1)^{i}\bigg((\xi_1^{i+1}-(-\xi_4)^{i+1})\sum_{m+l=j}\xi_2^m(-\xi_4)^l+\xi_4^{i+1}\sum_{m+l=j}\xi_3^m(-\xi_1)^l-(-\xi_2)^m\xi_4^l\bigg).
\end{align}
There are two terms on the right hand side of \eqref{equ-decomp2-16}. On one hand,
\begin{multline}\label{equ-decomp2-17}
\sum_{\substack{i+j=2k-2, \\ j\geq 1}}(-1)^{i}(\xi_1^{i+1}-(-\xi_4)^{i+1})\sum_{m+l=j}\xi_2^m(-\xi_4)^l\\ = \sum_{\substack{i+j=2k-2, \\ j\geq 1}}(-1)^{i}(\xi_1+\xi_4)\sum_{n+h=i}\xi_1^n(-\xi_4)^h\sum_{m+l=j}\xi_2^m(-\xi_4)^l.
\end{multline}
On the other hand,
\begin{align}\label{equ-decomp2-18}
\sum_{\substack{i+j=2k-2, \\ j\geq 1}}&(-1)^{i}\xi_4^{i+1}\sum_{m+l=j}\Big(\xi_3^m(-\xi_1)^l-(-\xi_2)^m\xi_4^l\Big)\nonumber\\
&=\sum_{\substack{i+j=2k-2, \\ j\geq 1}}(-1)^{i}\xi_4^{i+1}\bigg(\xi_3^j-(-\xi_2)^j + (-\xi_1)^j-\xi_4^j\bigg)\nonumber\\
& \quad +\sum_{\substack{i+j=2k-2, \\ j\geq 1}}(-1)^{i}\xi_4^{i+1}\sum_{\substack{m+l=j,\\ m,l\geq 1}}\bigg((-\xi_1)^l\Big(\xi_3^m-(-\xi_2)^m\Big)+(-\xi_2)^m\Big((-\xi_1)^l-\xi_4^l\Big)\bigg)\nonumber\\
&=(\xi_1+\xi_4)\sum_{\substack{i+j=2k-2, \\ j\geq 1}}(-1)^{i+1}\xi_4^{i+1}\bigg[\sum_{m+l=j-1}\xi_3^m(-\xi_2)^l+\xi_4^m(-\xi_1)^l\nonumber\\
&\quad +\sum_{\substack{m+l=j,\\ m,l\geq 1}}\Big((-\xi_1)^l \sum_{n+h=m-1}\xi_3^n(-\xi_2)^h + (-\xi_2)^m\sum_{n+h=l-1}\xi_4^n(-\xi_1)^h\Big) \bigg].
\end{align}
Then combining \eqref{equ-decomp2-15}-\eqref{equ-decomp2-18} gives \eqref{equ-decomp2-11}. The desired conclusion follows from \eqref{equ-decomp2-10} and \eqref{equ-decomp2-11}.

\subsection{Estimates for $\beta_4$}\label{subsec-3.3}
In this subsection, we give two upper bounds for $\beta_4$, based on the analysis of $M_4$ in subsection \ref{subsec-3.1} and \ref{subsec-3.2}.
Since $\beta_4 = -M_4/\alpha_4$ (see \eqref{equ-energy-4}), using Lemma \ref{lem-decomp-1} and Lemma \ref{lem-decomp-2}, we obtain
\begin{align}\label{equ-beta4-1}
\beta_4= \frac{c}{108}\sum_{k=0}^\infty \frac{\sigma^{2(k+2)}}{(2(k+2))!}(\Omega_1(k+2) -\Omega_2(k+1)),
\end{align}
where $\Omega_1(k+2):=\Omega_1(k+2; \xi_1,\cdots,\xi_4), \Omega_2(k+1):=\Omega_2(k+1; \xi_1,\cdots,\xi_4)$ are given by
\begin{equation}\label{equ-theta1-0}
\begin{split}
\Omega_1(k+2) & =\sum_{i+j=2k-1}(-1)^i \bigg((\xi_1^{j+1}+\xi_2^{j+1})\sum_{m+l=i}\xi_3^m(\xi_3+\xi_4)^l +(\xi_1^{i+1}-\xi_3^{i+1}) \sum_{m+l=j}\xi_2^m(\xi_2+\xi_4)^l\\
&\quad + (\xi_2^{i+1}-\xi_3^{i+1})\sum_{m+l=j}\xi_1^m(\xi_1+\xi_4)^l +(\xi_1^{i+1}+\xi_2^{i+1})\sum_{\substack{m+l=j}}\xi_4^{m}(\xi_3+\xi_4)^{l}\bigg)\\
&\quad -2 \sum_{i+j=2k}\bigg(\xi_1^i(\xi_1+\xi_4)^j+\xi_2^i(\xi_2+\xi_4)^j+ \xi_3^i(\xi_3+\xi_4)^j +\xi_4^{i}(\xi_3+\xi_4)^j\bigg)\\
& \quad -\sum_{i+j=2k}(-1)^i\bigg(\xi_4^{i}  \sum_{\substack{m+l=j}} \xi_1^{m}(\xi_1+\xi_3)^l +\xi_3^{i}  \sum_{\substack{m+l=j}} \xi_4^{m}(\xi_2+\xi_4)^l\bigg)\\
&\quad -\sum_{i+j=2k}(-1)^i\bigg(\xi_4^{i}  \sum_{\substack{m+l=j}} (\xi_2^{m}+\xi_3^{m})(\xi_2+\xi_3)^l \bigg),
\end{split}
\end{equation}
\begin{equation}\label{equ-theta2-0}
\begin{split}
\Omega_2(k+1)= & \sum_{i+j=2k}(-1)^i(2\xi_1^i\xi_4^j+\xi_2^i\xi_3^j)\\
& + \sum_{\substack{i+j=2k+1, \\ i,j\geq 1}}\bigg((-\xi_3)^i\sum_{m+l=j-1}\xi_1^m(-\xi_4)^l + \xi_4^j\sum_{m+l=i-1}(-\xi_2)^m\xi_3^l\bigg)\\
& +\sum_{\substack{i+j=2k, \\ j\geq 1}}(-1)^{i}\sum_{n+h=i}\xi_1^n(-\xi_4)^h\sum_{m+l=j}\xi_2^m(-\xi_4)^l \\
& +\sum_{\substack{i+j=2k, \\ j\geq 1}}(-1)^{i+1}\xi_4^{i+1}\bigg[\sum_{m+l=j-1}\Big(\xi_3^m(-\xi_2)^l+\xi_4^m(-\xi_1)^l\Big)\\
& +\sum_{\substack{m+l=j,\\ m,l\geq 1}}\Big((-\xi_1)^l \sum_{n+h=m-1}\xi_3^n(-\xi_2)^h + (-\xi_2)^m\sum_{n+h=l-1}\xi_4^n(-\xi_1)^h\Big) \bigg].
\end{split}
\end{equation}

Taking absolute value on both sides of \eqref{equ-theta1-0}, we find
\begin{align}\label{equ-theta1-1}
|\Omega_1(k+2)| & \leq \sum_{i+j=2k} \bigg((|\xi_1|^{j}+|\xi_2|^{j})\sum_{m+l=i}|\xi_3|^m|\xi_3+\xi_4|^l +(|\xi_1|^{i}+|\xi_3|^{i}) \sum_{m+l=j}|\xi_2|^m|\xi_2+\xi_4|^l\nonumber\\
&\quad + (|\xi_2|^{i}+|\xi_3|^{i})\sum_{m+l=j}|\xi_1|^m|\xi_1+\xi_4|^l +(|\xi_1|^{i}+|\xi_2|^{i})\sum_{\substack{m+l=j}}|\xi_4|^{m}|\xi_3+\xi_4|^{l}\bigg)\nonumber\\
&\quad +2 \sum_{i+j=2k}\bigg(|\xi_1|^i|\xi_1+\xi_4|^j+|\xi_2|^i|\xi_2+\xi_4|^j+ |\xi_3|^i|\xi_3+\xi_4|^j +|\xi_4|^{i}|\xi_3+\xi_4|^j\bigg)\nonumber\\
& \quad +\sum_{i+j=2k}\bigg(|\xi_4|^{i}  \sum_{\substack{m+l=j}} |\xi_1|^{m}|\xi_1+\xi_3|^l +|\xi_3|^{i}  \sum_{\substack{m+l=j}} |\xi_4|^{m}|\xi_2+\xi_4|^l\bigg)\nonumber\\
&\quad +\sum_{i+j=2k}\bigg(|\xi_4|^{i}  \sum_{\substack{m+l=j}} (|\xi_2|^{m}+|\xi_3|^{m})|\xi_2+\xi_3|^l \bigg)\nonumber\\
&\leq 2 \sum_{i+j=2k}\bigg(|\xi_1|^i|\xi_1+\xi_4|^j+|\xi_2|^i|\xi_2+\xi_4|^j+ |\xi_3|^i|\xi_3+\xi_4|^j +|\xi_4|^{i}|\xi_3+\xi_4|^j\bigg)\nonumber\\
& \quad + \sum_{i+j=2k} \sum {}^{\prime}|\xi_{p_1}|^i \sum_{m+l=j}|\xi_{p_2}|^m|\xi_{p_2}+\xi_{p_3}|^l.
\end{align}
Here, the sum $\sum{}^\prime$ is taking over all $p_1, p_2, p_3$ being different numbers in the set $\{1,2,3,4\}$.

Taking absolute value on both sides of \eqref{equ-theta2-0}, we find
\begin{equation}\label{equ-theta2-1}
\begin{split}
|\Omega_2(k+1)|\leq & \quad 2\sum_{i+j=2k}(|\xi_1|^i|\xi_4|^j+|\xi_2|^i|\xi_3|^j)+\sum_{i+j=2k}|\xi_3|^i\sum_{m+l=j}|\xi_1|^m|\xi_4|^l \\
& + 2\sum_{i+j=2k} |\xi_4|^j\sum_{m+l=i}|\xi_2|^m|\xi_3|^l+ \sum_{i+j=2k}|\xi_4|^{i}\sum_{m+l=j}|\xi_4|^m|\xi_1|^l\\
& +\sum_{i+j=2k}\sum_{n+h=i}|\xi_1|^n|\xi_4|^h\sum_{m+l=j}|\xi_2|^m|\xi_4|^l \\
& +\sum_{i+j=2k}|\xi_4|^{i}\bigg[\sum_{m+l=j}\Big(|\xi_1|^l \sum_{n+h=m}|\xi_3|^n|\xi_2|^h + |\xi_2|^m\sum_{n+h=l}|\xi_4|^n|\xi_1|^h\Big) \bigg].
\end{split}
\end{equation}

Thanks to \eqref{equ-beta4-1}, \eqref{equ-theta1-1}, \eqref{equ-theta2-1},  we obtain the following lemma.
\begin{lemma}\label{lem-beta4-1}
We have the following bound for $\beta_4$:
\begin{align}\label{equ-beta4-2}
|\beta_4| \leq \frac{|c|}{54}\sum_{k=0}^\infty \frac{\sigma^{k+4}}{(k+4)!}(\Theta_1(k)+\Theta_2(k)),
\end{align}
where $c$ is the constant in Lemma \ref{lem-M4}, and $\Theta_1(k), \Theta_2(k)$ are given by
\begin{equation}\label{equ-theta1-3}
\begin{split}
\,\,\,\,\,\Theta_1(k) = & \sum_{i+j=k}\bigg(|\xi_1|^i|\xi_1+\xi_4|^j+|\xi_2|^i|\xi_2+\xi_4|^j+ |\xi_3|^i|\xi_3+\xi_4|^j +|\xi_4|^{i}|\xi_3+\xi_4|^j\bigg)\\
& + \sum_{i+j=k} \sum {}^\prime|\xi_{p_1}|^i \sum_{m+l=j}|\xi_{p_2}|^m|\xi_{p_2}+\xi_{p_3}|^l,
\end{split}
\end{equation}
\begin{equation}\label{equ-theta2-3}
\begin{split}
\Theta_2(k) = &\sum_{i+j=k}(|\xi_1|^i|\xi_4|^j+|\xi_2|^i|\xi_3|^j)+ \sum_{i+j=k}|\xi_3|^i\sum_{m+l=j}|\xi_1|^m|\xi_4|^l \\
& + \sum_{i+j=k} |\xi_4|^j\sum_{m+l=i}|\xi_2|^m|\xi_3|^l+ \sum_{i+j=k}|\xi_4|^{i}\sum_{m+l=j}|\xi_4|^m|\xi_1|^l\\
& +\sum_{i+j=k}\sum_{n+h=i}|\xi_1|^n|\xi_4|^h\sum_{m+l=j}|\xi_2|^m|\xi_4|^l \\
& +\sum_{i+j=k}|\xi_4|^{i}\bigg[\sum_{m+l=j}\Big(|\xi_1|^l \sum_{n+h=m}|\xi_3|^n|\xi_2|^h + |\xi_2|^m\sum_{n+h=l}|\xi_4|^n|\xi_1|^h\Big) \bigg].
\end{split}
\end{equation}
\end{lemma}

To obtain further estimates of $\beta_4$, we need the following lemma.
\begin{lemma}\label{lem-gamma}
Let $p\geq 4$ be an integer. Assume that $n_1, n_2, \cdots, n_p$ are integers satisfying
$$
0\leq n_1\leq n_2, \cdots, n_{p-1} \leq n_{p}
$$
and
$$
n_1+n_p = n_2 + \cdots + n_{p-1}.
$$
Then we have
\begin{align}\label{equ-gamma-1}
n_2!n_3!\cdots n_{p-1}!\leq n_1!n_p!.
\end{align}
\end{lemma}
\begin{proof}
Without loss of generality, we assume that $n_2\leq n_3\leq \cdots \leq n_{p-1}$.  We divide the proof into two cases: $p=4$ and $p\geq 5$.

\textbf{The case $p=4$.} Using the relation $n!=\Gamma(n+1)$, where $\Gamma: \mathbb{R}^+ \mapsto \mathbb{R}$ is the standard Gamma function given by
\begin{align}\label{equ-gamma-2}
\Gamma(x)=\int_0^\infty t^{x-1}e^{-t}dt, \quad x>0,
\end{align}
the conclusion \eqref{equ-gamma-1} can be restated as
\begin{align}\label{equ-gamma-3}
\Gamma(n_2+1)\Gamma(n_3+1)\leq\Gamma(n_1+1)\Gamma(n_4+1).
\end{align}
It remains to prove \eqref{equ-gamma-3}. Using H\"{o}lder inequality, we deduce from \eqref{equ-gamma-2} that
\begin{align}\label{equ-gamma-4}
\Gamma(\theta x_1+(1-\theta)x_2)\leq \Gamma(x_1)^\theta\Gamma(x_2)^{1-\theta}, \quad \theta\in [0,1], x_1,x_2>0.
\end{align}
Define $f(x):=\ln \Gamma(x), x>0$. Then \eqref{equ-gamma-4} implies that $f$ is convex. Since $\Gamma$ is smooth, so is $f$. Then $f''(x)\geq 0, x>0$. We claim that
\begin{align}\label{equ-gamma-5}
f(b)+f(c)\leq f(a)+f(d), \quad 0<a\leq b\leq c \leq d,\quad b+c=a+d.
\end{align}
In fact, by mean value theorem we have for some $\eta_1\in [a,b], \eta_2\in [c,d]$
$$
f(b)-f(a)=f'(\eta_1)(b-a), \quad f(d)-f(c)=f'(\eta_2)(d-c).
$$
From this, we use mean value theorem again to find for some $\eta_3\in [\eta_1,\eta_2]$
\begin{align*}
 f(a)+f(d)-(f(b)+f(c)) = (f'(\eta_2)-f'(\eta_1))(b-a)=f''(\eta_3)(\eta_2-\eta_1)(b-a)\geq 0.
\end{align*}
Thus the claim \eqref{equ-gamma-5} follows. Set
$$
a=n_1+1, b=n_2+1, c=n_3+1, d=n_4+1
$$
in \eqref{equ-gamma-5}, we obtain \eqref{equ-gamma-3}.

\textbf{The case $p\geq 5$.} Clearly, we have $m!n!\leq (m+n)!$ for all nonnegative integers $m,n$. Using the fact repeatedly, we find
$$
n_3!\cdots n_{p-1}!\leq (n_3+\cdots+  n_{p-1})!.
$$
Thus the desired conclusion \eqref{equ-gamma-1} holds if one can show
\begin{align*}
n_2!(n_3+\cdots+  n_{p-1})!\leq n_1!n_p!.
\end{align*}
But this follows from the proved case $p=4$.
\end{proof}

We are ready to state our first bound for $\beta_4$.
\begin{lemma}\label{lem-beta4-2}
Let $\beta_4$ be given by \eqref{equ-beta4-1}. Then we have for all $\sigma \geq 0, \xi_1+\xi_2+\xi_3+\xi_4=0$
\begin{align}\label{equ-beta4-3}
|\beta_4| \leq \frac{43|c|}{54}\sigma^4e^{\sigma(|\xi_1|+|\xi_2|+|\xi_3|+|\xi_4|)}.
\end{align}
\end{lemma}

\begin{proof}
Thanks to \eqref{equ-beta4-2}, the conclusion \eqref{equ-beta4-3} follows from the following two inequalities:
\begin{align}\label{equ-theta1-4}
\sum_{k=0}^\infty \frac{\sigma^{k+4}}{(k+4)!}\Theta_1(k) \leq 28\sigma^4e^{\sigma(|\xi_1|+|\xi_2|+|\xi_3|+|\xi_4|)},
\end{align}
\begin{align}\label{equ-theta2-4}
\sum_{k=0}^\infty \frac{\sigma^{k+4}}{(k+4)!}\Theta_2(k) \leq 15\sigma^4e^{\sigma(|\xi_1|+|\xi_2|+|\xi_3|+|\xi_4|)}.
\end{align}

\textbf{The proof of \eqref{equ-theta1-4}.} It suffices to show that
\begin{multline}\label{equ-theta1-5}
\sum_{k=0}^\infty \frac{\sigma^{k+4}}{(k+4)!}\sum_{i+j=k}\bigg(|\xi_1|^i|\xi_1+\xi_4|^j+|\xi_2|^i|\xi_2+\xi_4|^j+ |\xi_3|^i|\xi_3+\xi_4|^j +|\xi_4|^{i}|\xi_3+\xi_4|^j\bigg)\\
\leq 4\sigma^4e^{\sigma(|\xi_1|+|\xi_2|+|\xi_3|+|\xi_4|)},
\end{multline}
\begin{align}\label{equ-theta1-6}
\sum_{k=0}^\infty \frac{\sigma^{k+4}}{(k+4)!}\sum_{i+j=k} \sum {}^\prime|\xi_{p_1}|^i \sum_{m+l=j}|\xi_{p_2}|^m|\xi_{p_2}+\xi_{p_3}|^l\leq 24\sigma^4e^{\sigma(|\xi_1|+|\xi_2|+|\xi_3|+|\xi_4|)}.
\end{align}
To prove \eqref{equ-theta1-5}, we first show that
\begin{align}\label{equ-theta1-7}
\sum_{k=0}^\infty \frac{\sigma^{k+4}}{(k+4)!}\sum_{i+j=k}|\xi_1|^i|\xi_1+\xi_4|^j\leq \sigma^4e^{\sigma(|\xi_1|+|\xi_4|)}.
\end{align}
In fact, we expand the left hand side of \eqref{equ-theta1-7} by the binomial theorem to find
\begin{align}\label{equ-theta1-8}
\sum_{k=0}^\infty \frac{\sigma^{k+4}}{(k+4)!}\sum_{i+j=k}|\xi_1|^i|\xi_1+\xi_4|^j\leq\sum_{k=0}^\infty \frac{\sigma^{k+4}}{(k+4)!}\sum_{i+j=k}\sum_{m+l=j}\frac{j!}{m!l!}|\xi_1|^{i+m}|\xi_4|^l.
\end{align}
If $i+j=k, m+l=j$, then Lemma \ref{lem-gamma} gives $\frac{j!}{m!k!}\leq \frac{1}{(m+i)!}$. Thus, we deduce from \eqref{equ-theta1-8} that
\begin{align}\label{equ-theta1-9}
\sum_{k=0}^\infty \frac{\sigma^{k+4}}{(k+4)!}\sum_{i+j=k}|\xi_1|^i|\xi_1+\xi_4|^j\leq\sum_{k=0}^\infty \frac{\sigma^{k+4}}{(k+1)^4}\sum_{i+j=k}\sum_{m+l=j}\frac{|\xi_1|^{i+m}}{(m+i)!}\frac{|\xi_4|^l}{l!}.
\end{align}
Since $\sum_{i+j=k}\sum_{m+l=j}\frac{|\xi_1|^{i+m}}{(m+i)!}\frac{|\xi_4|^l}{l!} \leq (k+1)\sum_{i+j=k}\frac{|\xi_1|^{i}}{i!}\frac{|\xi_4|^j}{j!}$, \eqref{equ-theta1-9} becomes
\begin{align}\label{equ-theta1-9.5}
\sum_{k=0}^\infty \frac{\sigma^{k+4}}{(k+4)!}\sum_{i+j=k}|\xi_1|^i|\xi_1+\xi_4|^j\leq\sum_{k=0}^\infty \frac{\sigma^{k+4}}{(k+1)^3}\sum_{i+j=k}\frac{|\xi_1|^{i}}{i!}\frac{|\xi_4|^j}{j!}\leq \sigma^4e^{\sigma(|\xi_1|+|\xi_4|)}.
\end{align}
This proves \eqref{equ-theta1-7}. Similarly, we have
\begin{multline}\label{equ-theta1-10}
\sum_{k=0}^\infty \frac{\sigma^{k+4}}{(k+4)!}\sum_{i+j=k}\bigg(|\xi_2|^i|\xi_2+\xi_4|^j+ |\xi_3|^i|\xi_3+\xi_4|^j +|\xi_4|^{i}|\xi_3+\xi_4|^j\bigg)\\
\leq \sigma^4\Big(e^{\sigma(|\xi_2|+|\xi_4|)} +  2e^{\sigma(|\xi_3|+|\xi_4|)}\Big).
\end{multline}
Combining \eqref{equ-theta1-7} and \eqref{equ-theta1-10} implies that \eqref{equ-theta1-5}.

To  prove \eqref{equ-theta1-6}, we first show that
\begin{align}\label{equ-theta1-11}
\sum_{k=0}^\infty \frac{\sigma^{k+4}}{(k+4)!}\sum_{i+j=k} |\xi_{1}|^i \sum_{m+l=j}|\xi_{2}|^m|\xi_{2}+\xi_{3}|^l\leq \sigma^4e^{\sigma(|\xi_1|+|\xi_2|+|\xi_3|+|\xi_4|)}.
\end{align}
The idea is similar to that of proving \eqref{equ-theta1-5}. In fact, using Lemma \ref{lem-gamma} we have
\begin{align}\label{equ-theta1-12}
\sum_{k=0}^\infty &\frac{\sigma^{k+4}}{(k+4)!}\sum_{i+j=k} |\xi_{1}|^i \sum_{m+l=j}|\xi_{2}|^m|\xi_{2}+\xi_{3}|^l\nonumber\\
&\leq \sum_{k=0}^\infty \frac{\sigma^{k+4}}{(k+4)!}\sum_{i+j=k} |\xi_{1}|^i \sum_{m+l=j}|\xi_{2}|^m \sum_{n+h=l}\frac{l!}{n!h!}|\xi_{2}|^n|\xi_{3}|^h\nonumber\\
&\leq \sum_{k=0}^\infty \frac{\sigma^{k+4}}{(k+1)^4}\sum_{i+j=k} \frac{|\xi_{1}|^i}{i!} \sum_{m+l=j} \sum_{n+h=l}\frac{|\xi_{2}|^{m+n}}{(m+n)!}\frac{|\xi_{3}|^h}{h!}\nonumber\\
&\leq \sum_{k=0}^\infty \frac{\sigma^{k+4}}{(k+1)^3}\sum_{i+j=k} \frac{|\xi_{1}|^i}{i!} \sum_{m+l=j}\frac{|\xi_{2}|^{m}}{m!}\frac{|\xi_{3}|^l}{l!}\nonumber\\
&\leq \sigma^4e^{\sigma(|\xi_1|+|\xi_2|+|\xi_3|)}.
\end{align}
Clearly, \eqref{equ-theta1-11} follows from \eqref{equ-theta1-12}. Similarly, we have for $p_1,p_2,p_3$ being different numbers in the set $\{1,2,3,4\}$
\begin{align}\label{equ-theta1-13}
\sum_{k=0}^\infty \frac{\sigma^{k+4}}{(k+4)!}\sum_{i+j=k} |\xi_{p_1}|^i \sum_{m+l=j}|\xi_{p_2}|^m|\xi_{p_2}+\xi_{p_3}|^l\leq \sigma^4e^{\sigma(|\xi_1|+|\xi_2|+|\xi_3|+|\xi_4|)}.
\end{align}
Since the number of different choices of $p_1, p_2, p_3$ is $24$, the conclusion \eqref{equ-theta1-6} follows from \eqref{equ-theta1-13}.

\textbf{The proof of \eqref{equ-theta2-4}.} It suffices to show that
\begin{align}\label{equ-theta1-14}
\sum_{k=0}^\infty \frac{\sigma^{k+4}}{(k+4)!}\sum_{i+j=k}(|\xi_1|^i|\xi_4|^j+|\xi_2|^i|\xi_3|^j) \leq 2\sigma^4e^{\sigma(|\xi_1|+|\xi_2|+|\xi_3|+|\xi_4|)},
\end{align}
\begin{multline}\label{equ-theta1-15}
\sum_{k=0}^\infty \frac{\sigma^{k+4}}{(k+4)!}\sum_{i+j=k} \bigg[|\xi_3|^i\sum_{m+l=j}|\xi_1|^m|\xi_4|^l+  |\xi_4|^j\sum_{m+l=i}|\xi_2|^m|\xi_3|^l\bigg]\leq 2\sigma^4e^{\sigma(|\xi_1|+|\xi_2|+|\xi_3|+|\xi_4|)},
\end{multline}
\begin{multline}\label{equ-theta1-16}
\sum_{k=0}^\infty \frac{\sigma^{k+4}}{(k+4)!}\sum_{i+j=k}|\xi_4|^{i}\bigg[\sum_{m+l=j}\Big(|\xi_1|^l \sum_{n+h=m}|\xi_3|^n|\xi_2|^h + |\xi_2|^m\sum_{n+h=l}|\xi_4|^n|\xi_1|^h\Big) \bigg]\\
\leq 2\sigma^4e^{\sigma(|\xi_1|+|\xi_2|+|\xi_3|+|\xi_4|)},
\end{multline}
\begin{align}\label{equ-theta1-16.5}
\sum_{k=0}^\infty \frac{\sigma^{k+4}}{(k+4)!}\sum_{i+j=k}  |\xi_4|^{i}\sum_{m+l=j}|\xi_4|^m|\xi_1|^l\leq \sigma^4e^{\sigma(|\xi_1|+|\xi_2|+|\xi_3|+|\xi_4|)},
\end{align}
\begin{align}\label{equ-theta1-17}
\sum_{k=0}^\infty \frac{\sigma^{k+4}}{(k+4)!}\sum_{i+j=k}\sum_{n+h=i}|\xi_1|^n|\xi_4|^h\sum_{m+l=j}|\xi_2|^m|\xi_4|^l\leq 4\sigma^4e^{\sigma(|\xi_1|+|\xi_2|+|\xi_3|+|\xi_4|)}.
\end{align}
Clearly, we have the following inequalities:
\begin{align}\label{equ-theta1-18}
\frac{1}{k!}\leq \frac{1}{i!}\frac{1}{j!}, \quad i+j=k,
\end{align}
\begin{align}\label{equ-theta1-19}
\frac{1}{k!}\leq \frac{1}{i!}\frac{1}{m!}\frac{1}{l!}, \quad i+j=k, m+l=j,
\end{align}
\begin{align}\label{equ-theta1-20}
\frac{1}{k!}\leq \frac{1}{i!}\frac{1}{l!}\frac{1}{n!}\frac{1}{h!}, \quad i+j=k,m+l=j, n+h=m.
\end{align}
Then the equalities \eqref{equ-theta1-14}-\eqref{equ-theta1-16} follows from \eqref{equ-theta1-18}-\eqref{equ-theta1-20}. The equality \eqref{equ-theta1-16.5} follows from \eqref{equ-theta1-7}.

It remains to prove \eqref{equ-theta1-17}. Indeed, using the elementary inequality $|\xi_1|^n|\xi_4|^h\leq |\xi_1|^i + |\xi_4|^i$ for all $n+h=i$, we find $\sum_{n+h=i}|\xi_1|^n|\xi_4|^h \leq i(|\xi_1|^i + |\xi_4|^i)$. Similarly, $\sum_{m+l=j}|\xi_2|^m|\xi_4|^l \leq j(|\xi_2|^j + |\xi_4|^j)$. Then we deduce that
\begin{align*}
\sum_{k=0}^\infty &\frac{\sigma^{k+4}}{(k+4)!}\sum_{i+j=k}\sum_{n+h=i}|\xi_1|^n|\xi_4|^h\sum_{m+l=j}|\xi_2|^m|\xi_4|^l\nonumber\\
&\leq \sum_{k=0}^\infty \frac{\sigma^{k+4}}{(k+4)!}\sum_{i+j=k}i j(|\xi_1|^i + |\xi_4|^i)(|\xi_2|^j + |\xi_4|^j)\nonumber\\
&\leq \sum_{k=0}^\infty \frac{\sigma^{k+4}}{(k+2)!} \sum_{i+j=k}(|\xi_1|^i|\xi_4|^j + |\xi_4|^i|\xi_2|^j + |\xi_1|^i|\xi_2|^j + |\xi_4|^k)\nonumber\\
&\leq \sum_{k=0}^\infty \sigma^{k+4} \bigg(\frac{1}{k!}|\xi_4|^k+ \sum_{i+j=k}\frac{1}{i!j!}(|\xi_1|^i|\xi_4|^j + |\xi_4|^i|\xi_2|^j + |\xi_1|^i|\xi_2|^j )\bigg)\nonumber\\
&\leq 4\sigma^4e^{\sigma(|\xi_1|+|\xi_2|+|\xi_3|+|\xi_4|)}.
\end{align*}
This proves \eqref{equ-theta1-17}.
\end{proof}

The second bound for $\beta_4$ is given as follows.
\begin{lemma}\label{lem-beta4-3}
Let $\beta_4$ be given by \eqref{equ-beta4-1}. Then we have for all $\sigma \leq 1, \xi_1+\xi_2+\xi_3+\xi_4=0$
\begin{align}\label{equ-beta4-lemma2}
|\beta_4| \leq \frac{|c|}{9}\sum_{p_1\neq p_2, p_1,p_2\in \{1,2,3,4\}}\frac{1}{(1+|\xi_{p_1}|)(1+|\xi_{p_2}|)}e^{\sigma(|\xi_1|+|\xi_2|+|\xi_3|+|\xi_4|)}.
\end{align}
\end{lemma}
\begin{proof}
Thanks to \eqref{equ-beta4-2}, it suffices to show the following two inequalities:
\begin{align}\label{equ-theta1-21}
\sum_{k=0}^\infty \frac{\sigma^{k+4}}{(k+4)!}\Theta_1(k) \leq 3\sum_{p_1\neq p_2, p_1,p_2\in \{1,2,3,4\}}\frac{1}{(1+|\xi_{p_1}|)(1+|\xi_{p_2}|)}e^{\sigma(|\xi_1|+|\xi_2|+|\xi_3|+|\xi_4|)},
\end{align}
\begin{align}\label{equ-theta1-22}
\sum_{k=0}^\infty \frac{\sigma^{k+4}}{(k+4)!}\Theta_2(k) \leq 3\sum_{p_1\neq p_2, p_1,p_2\in \{1,2,3,4\}}\frac{1}{(1+|\xi_{p_1}|)(1+|\xi_{p_2}|)}e^{\sigma(|\xi_1|+|\xi_2|+|\xi_3|+|\xi_4|)}.
\end{align}

\textbf{The proof of \eqref{equ-theta1-21}.} It suffices to prove that
\begin{multline}\label{equ-theta1-23}
\sum_{k=0}^\infty \frac{\sigma^{k+4}}{(k+4)!}\sum_{i+j=k}\bigg(|\xi_1|^i|\xi_1+\xi_4|^j+|\xi_2|^i|\xi_2+\xi_4|^j+ |\xi_3|^i|\xi_3+\xi_4|^j +|\xi_4|^{i}|\xi_3+\xi_4|^j\bigg)\\
\leq \sum_{p_1\neq p_2, p_1,p_2\in \{1,2,3,4\}}\frac{1}{(1+|\xi_{p_1}|)(1+|\xi_{p_2}|)}e^{\sigma(|\xi_1|+|\xi_2|+|\xi_3|+|\xi_4|)},
\end{multline}
\begin{multline}\label{equ-theta1-24}
\sum_{k=0}^\infty \frac{\sigma^{k+4}}{(k+4)!}\sum_{i+j=k} \sum '|\xi_{p_1}|^i \sum_{m+l=j}|\xi_{p_2}|^m|\xi_{p_2}+\xi_{p_3}|^l\\
\leq 2\sum_{p_1\neq p_2, p_1,p_2\in \{1,2,3,4\}}\frac{1}{(1+|\xi_{p_1}|)(1+|\xi_{p_2}|)}e^{\sigma(|\xi_1|+|\xi_2|+|\xi_3|+|\xi_4|)}.
\end{multline}

To prove \eqref{equ-theta1-23}, inserting the inequality $\sigma(1+|\xi|)\leq e^{\sigma|\xi|}$ for $0\leq \sigma \leq 1$ into \eqref{equ-theta1-7}, we find
\begin{align*}
\sum_{k=0}^\infty \frac{\sigma^{k+4}}{(k+4)!}\sum_{i+j=k}|\xi_1|^i|\xi_1+\xi_4|^j\leq \frac{1}{(1+|\xi_2|)(1+|\xi_3|)}e^{\sigma(|\xi_1|+|\xi_2|+|\xi_3|+|\xi_4|)}.
\end{align*}
Similarly, the other terms on the left hand side of \eqref{equ-theta1-23} can be bounded. This proves \eqref{equ-theta1-23}.

To prove \eqref{equ-theta1-24}, it suffices to establish that if $p_1, p_2, p_3, p_4$ is a permutation of $1,2,3,4$, then
\begin{multline}\label{equ-theta1-25}
\sum_{k=0}^\infty \frac{\sigma^{k+4}}{(k+4)!}\sum_{i+j=k} |\xi_{p_1}|^i \sum_{m+l=j}|\xi_{p_2}|^m|\xi_{p_2}+\xi_{p_3}|^l\leq \frac{1}{(1+|\xi_{p_1}|)(1+|\xi_{p_4}|)}e^{\sigma(|\xi_1|+|\xi_2|+|\xi_3|+|\xi_4|)}.
\end{multline}
(Note that the factor $2$ on the right side of \eqref{equ-theta1-24} is needed, if one considers the number of terms for two sums in \eqref{equ-theta1-24}.) In fact, on one hand, thanks to \eqref{equ-theta1-12},
\begin{align}\label{equ-theta1-26}
\sum_{k=0}^\infty \frac{\sigma^{k+4}}{(k+4)!}\sum_{i+j=k} |\xi_{p_1}|^i \sum_{m+l=j}|\xi_{p_2}|^m|\xi_{p_2}+\xi_{p_3}|^l\leq \sigma^4e^{\sigma(|\xi_{p_1}|+|\xi_{p_2}|+|\xi_{p_3}|)}.
\end{align}
On the other hand, similar to the proof of \eqref{equ-theta1-12}
\begin{align}\label{equ-theta1-27}
\sum_{k=0}^\infty \frac{\sigma^{k+4}}{(k+4)!}\sum_{i+j=k} |\xi_{p_1}|^{i+1} \sum_{m+l=j}|\xi_{p_2}|^m|\xi_{p_2}+\xi_{p_3}|^l&\leq \sum_{k=0}^\infty \frac{\sigma^{k+3}}{(k+3)!}\sum_{i+j=k} |\xi_{p_1}|^{i} \sum_{m+l=j}|\xi_{p_2}|^m|\xi_{p_2}+\xi_{p_3}|^l\nonumber\\
&\leq \sigma^3e^{\sigma(|\xi_{p_1}|+|\xi_{p_2}|+|\xi_{p_3}|)}.
\end{align}
Combining \eqref{equ-theta1-26} and \eqref{equ-theta1-27} gives
\begin{align}\label{equ-theta1-28}
(1+|\xi_{p_1}|)\sum_{k=0}^\infty \frac{\sigma^{k+4}}{(k+4)!}\sum_{i+j=k} |\xi_{p_1}|^i \sum_{m+l=j}|\xi_{p_2}|^m|\xi_{p_2}+\xi_{p_3}|^l\leq \sigma^3e^{\sigma(|\xi_{p_1}|+|\xi_{p_2}|+|\xi_{p_3}|)}.
\end{align}
Using $\sigma(1+|\xi|)\leq e^{\sigma|\xi|}$ for $0\leq \sigma \leq 1$ again, the inequality \eqref{equ-theta1-25} follows from \eqref{equ-theta1-28}.

\textbf{The proof of \eqref{equ-theta1-22}.} Using the idea of the proof of \eqref{equ-theta1-21}, we obtain the following estimates:
\begin{multline}\label{equ-theta1-29}
\sum_{k=0}^\infty \frac{\sigma^{k+4}}{(k+4)!}\sum_{i+j=k}(|\xi_1|^i|\xi_4|^j+|\xi_2|^i|\xi_3|^j)\\
 \leq \bigg(\frac{1}{(1+|\xi_1|)(1+|\xi_4|)}+\frac{1}{(1+|\xi_{2}|)(1+|\xi_{3}|)}\bigg)e^{\sigma(|\xi_1|+|\xi_2|+|\xi_3|+|\xi_4|)},
\end{multline}
\begin{multline}\label{equ-theta1-30}
\sum_{k=0}^\infty \frac{\sigma^{k+4}}{(k+4)!}\sum_{i+j=k} \bigg[|\xi_3|^i\sum_{m+l=j}|\xi_1|^m|\xi_4|^l+  |\xi_4|^j\sum_{m+l=i}|\xi_2|^m|\xi_3|^l + |\xi_4|^{i}\sum_{m+l=j}|\xi_4|^m|\xi_1|^l\bigg]\\
\leq \bigg(\frac{2}{(1+|\xi_1|)(1+|\xi_4|)}+\frac{1}{(1+|\xi_{2}|)(1+|\xi_{3}|)}\bigg)e^{\sigma(|\xi_1|+|\xi_2|+|\xi_3|+|\xi_4|)},
\end{multline}
\begin{multline}\label{equ-theta1-31}
\sum_{k=0}^\infty \frac{\sigma^{k+4}}{(k+4)!}\sum_{i+j=k}|\xi_4|^{i}\bigg[\sum_{m+l=j}\Big(|\xi_1|^l \sum_{n+h=m}|\xi_3|^n|\xi_2|^h + |\xi_2|^m\sum_{n+h=l}|\xi_4|^n|\xi_1|^h\Big) \bigg]\\
\leq \bigg(\frac{1}{(1+|\xi_1|)(1+|\xi_4|)}+\frac{1}{(1+|\xi_{2}|)(1+|\xi_{4}|)}\bigg)e^{\sigma(|\xi_1|+|\xi_2|+|\xi_3|+|\xi_4|)},
\end{multline}
\begin{multline}\label{equ-theta1-32}
\sum_{k=0}^\infty \frac{\sigma^{k+4}}{(k+4)!}\sum_{i+j=k}\sum_{n+h=i}|\xi_1|^n|\xi_4|^h\sum_{m+l=j}|\xi_2|^m|\xi_4|^l\\
\leq \left(\frac{2}{(1+|\xi_1|)(1+|\xi_2|)}+\Big(\frac{1}{1+|\xi_1|}+\frac{1}{1+|\xi_2|}\Big)\frac{1}{1+|\xi_4|}\right)e^{\sigma(|\xi_1|+|\xi_2|+|\xi_3|+|\xi_4|)}.
\end{multline}
Combining \eqref{equ-theta1-29}-\eqref{equ-theta1-32} implies \eqref{equ-theta1-22}.
\end{proof}

\section{The analyticity radius for KdV}

In this section, we first shall prove an almost conservation law for the KdV equation \ref{equ-1} in  Gevrey class spaces, based on the upper bounds in the subsection \ref{subsec-3.3}. Then we using the almost conversation law and an iteration argument to prove Theorem \ref{thm1}.

\subsection{Almost conservation law}
Recall that the energy  $E^2_I(t)=\|Iu\|^2_{L^2}$, see \eqref{equ-energy-2}. The following lemma shows that, for every $t\in \mathbb{R}$, the energy $E^4_I(t)$ is comparable to $E^2_I(t)$ if $\|Iu\|_{L^2}$ is small.
\begin{lemma}\label{lem-comp}
Let $I$ be the operator defined with the Fourier symbol $m$ given by \eqref{equ-m}, $0<\sigma\leq 1$. Then there exists an absolute constant $C$ such that  for $t\in \mathbb{R}$
\begin{align}\label{equ-comp-1}
|E^4_I(t)-E^2_I(t)|\leq C(\|Iu\|^3_{L^2}+\|Iu\|^4_{L^2}).
\end{align}
\end{lemma}
\begin{proof}
Since $E^4_I(t)=E^2_I(t)+\Lambda_3(\beta_3; u, u, u) + \Lambda_4(\beta_4; u, u, u, u)$, it suffices to show
\begin{align}\label{equ-comp-2}
|\Lambda_3(\beta_3; u, u, u)|\leq \|Iu\|^3_{L^2},
\end{align}
\begin{align}\label{equ-comp-3}
|\Lambda_4(\beta_4; u, u, u, u)|\leq \|Iu\|^4_{L^2}.
\end{align}
Without loss of generality, we assume that $\widehat{u}$ is nonnegative.

\textbf{Proof of \eqref{equ-comp-3}.} According to Lemma \ref{lem-beta4-3}, using the property of Fourier transform, we find
\begin{align}\label{equ-comp-4}
|\Lambda_4(\beta_4; u, u, u, u)|&\leq \int_{\xi_1+\xi_2+\xi_3+\xi_4=0}\frac{|c|}{9}\sum_{p_1\neq p_2, p_1,p_2\in \{1,2,3,4\}}\frac{1}{(1+|\xi_{p_1}|)(1+|\xi_{p_2}|)}\prod_{i=1}^4 e^{\sigma|\xi_i|}u(\xi_i)\nonumber\\
&\leq \frac{4|c|}{3}\int_{\mathbb{R}}|\mathcal {F}^{-1}(\frac{1}{1+|\xi|})e^{\sigma |\xi|}\widehat{u}(\xi)|^2|\mathcal {F}^{-1}e^{\sigma|\xi|}\widehat{u}(\xi)|^2 dx,
\end{align}
where $\mathcal {F}^{-1}$ denotes the inverse Fourier transform.  Using the Sobolev embedding $H^{1}(\mathbb{R})\hookrightarrow L^\infty(\mathbb{R})$, we derive from \eqref{equ-comp-4} that
\begin{align*}
|\Lambda_4(\beta_4; u, u, u, u)|&\leq  \frac{4|c|}{3}\|\mathcal {F}^{-1}(\frac{1}{1+|\xi|})e^{\sigma |\xi|}\widehat{u}(\xi)\|^2_{L^\infty} \|\mathcal {F}^{-1}e^{\sigma|\xi|}\widehat{u}(\xi)\|^2_{L^2}\\
&\leq  C\|e^{\sigma |\xi|}\widehat{u}(\xi)\|^4_{L^2}\leq 2^4C\|Iu\|^4_{L^2}.
\end{align*}
This proves \eqref{equ-comp-3}.

\textbf{Proof of \eqref{equ-comp-2}.} The idea is similar to \eqref{equ-comp-3}. We only give a sketch. If $1\leq k\in \mathbb{N}$ and $\xi_1+\xi_2+\xi_3=0$, then one can show that
\begin{align}\label{equ-comp-5}
\xi_1^{2k+1}+\xi_2^{2k+1}+\xi_3^{2k+1}=\xi_1\xi_2\xi_3\sum_{i+j=2k-2} \bigg(\xi_3^j\Big((-\xi_1)^i+(-\xi_2)^i \Big) +\xi_1^i(-\xi_2)^j \bigg).
\end{align}
In particular, this gives $\alpha_3=i(\xi_1^3+\xi_2^3+\xi_3^3)=3i\xi_1\xi_2\xi_3$ for $\xi_1+\xi_2+\xi_3=0$. Recall that $\beta_3=i[m(\xi_1)m(\xi_2+\xi_3)\{\xi_2+\xi_3\}]_{sym}/\alpha_3$, by \eqref{equ-taylor} we find
\begin{align}\label{equ-comp-6}
\beta_3=-\frac{1}{9\xi_1\xi_2\xi_3}\sum_{k=1}^\infty \frac{\sigma^{2k}}{(2k)!}(\xi_1^{2k+1}+\xi_2^{2k+1}+\xi_3^{2k+1}).
\end{align}
Combining \eqref{equ-comp-5} and \eqref{equ-comp-6} gives
$$
\beta_3=-\frac{1}{9}\sum_{k=1}^\infty \frac{\sigma^{2k}}{(2k)!}\sum_{i+j=2k-2} \bigg(\xi_3^j\Big((-\xi_1)^i +(-\xi_2)^i\Big) +\xi_1^i(-\xi_2)^j \bigg).
$$
From this, one can show that for $0<\sigma\leq 1$
\begin{align}\label{equ-comp-7}
|\beta_3|\leq \sum_{i=1}^3\frac{1}{1+|\xi_i|}e^{\sigma(|\xi_1|+|\xi_2|+|\xi_3|)}.
\end{align}
Then \eqref{equ-comp-2} follows from \eqref{equ-comp-7}.
\end{proof}

\begin{lemma}\label{lem-app}
Let $I$ be the operator defined with the Fourier symbol $m$ given by \eqref{equ-m}, $0<\sigma\leq 1$. Then for $b\in (\frac{1}{2},\frac{2}{3})$ there exists a constant $C=C(b)$ such that
\begin{align}\label{equ-app-1}
\left|\int_0^\delta \Lambda_5(M_5;u,u,u,u,u)\,\mathrm dt \right|\leq C\sigma^4\|Iu\|^5_{X^{0,b}_\delta}.
\end{align}
\end{lemma}
\begin{proof}
Without loss of generality, we assume that $\widehat{u}$ is nonnegative again.
Recall that (see \eqref{equ-beta4-0})
$$
M_5(\xi_1, \xi_2, \xi_3,\xi_4,\xi_5)=-2i[\beta_4(\xi_1,\xi_2,\xi_3,\xi_4+\xi_5)\{\xi_4+\xi_5\}]_{sym},
$$
using Lemma \ref{lem-beta4-2}, we find
\begin{align}\label{equ-app-2}
|M_5| \leq \frac{86|c|}{27}\sigma^4e^{\sigma(|\xi_1|+|\xi_2|+|\xi_3|+|\xi_4|+|\xi_5|)}\max_{i=1,2,\cdots,5}|\xi_i|.
\end{align}
We first use the bound \eqref{equ-app-2}, and then Parseval identity to obtain that
\begin{align}\label{equ-app-3}
\left|\int_0^\delta \Lambda_5(M_5;u,u,u,u,u)\,\mathrm dt \right| &\leq \frac{86|c|}{27}\sigma^4\int_0^\delta\int_{\xi_1+\xi_2+\cdots+\xi_5=0}\max_{i=1,2,\cdots,5}|\xi_i|\prod_{i=1}^5e^{\sigma|\xi_i|}\widehat{u}(\xi_i) \,\mathrm dt\nonumber\\
&=  \frac{86|c|}{27}\sigma^4\int_0^\delta\int_{\mathbb{R}} |D|e^{\sigma |D|}u \cdot (e^{\sigma |D|}u)^4 \,\mathrm dx \mathrm dt,
\end{align}
where $|D|$ and $e^{\sigma|D|}$ are the Fourier multiplier with symbol $|\xi|$ and $e^{\sigma|\xi|}$, respectively. The integral on right hand side of \eqref{equ-app-3} can be bounded by
\begin{align}\label{equ-app-4}
\int_0^\delta\int_{\mathbb{R}} e^{\sigma |D|}u \cdot |D|(e^{\sigma |D|}u)^4\,\mathrm dx \mathrm dt \leq \|e^{\sigma |D|}u\|_{X_\delta^{0,1-b}}\||D|(e^{\sigma |D|}u)^4\|_{X_\delta^{0,b-1}}
\end{align}
for all $b\in (\frac{1}{2},1)$. Applying Lemma \ref{lem1} with $u_i=e^{\sigma |D|}u, i=1,2,3,4$, $b'=b-1$, we find for $b\in (\frac{1}{2},\frac{2}{3})$
\begin{align}\label{equ-app-5}
\||D|(e^{\sigma |D|}u)^4\|_{X_\delta^{0,b-1}}\leq C\|e^{\sigma |D|}u\|^4_{X_\delta^{0,b}},
\end{align}
where $C$ is a constant depends only on $b$. Note that $1-b<b$ (since $b>\frac{1}{2}$), we have $\|e^{\sigma |D|}u\|_{X_\delta^{0,1-b}}\leq \|e^{\sigma |D|}u\|_{X_\delta^{0,b}}$. Inserting \eqref{equ-app-5} into \eqref{equ-app-4}, we obtain
\begin{align}\label{equ-app-6}
\int_0^\delta\int_{\mathbb{R}} e^{\sigma |D|}u \cdot |D|(e^{\sigma |D|}u)^4 \,\mathrm dx \mathrm dt \leq C\|e^{\sigma |D|}u\|_{X_\delta^{0,b}}^5.
\end{align}
Combining \eqref{equ-app-3} and \eqref{equ-app-6} we get
\begin{align*}
\left|\int_0^\delta \Lambda_5(M_5;u,u,u,u,u)\, \mathrm dt \right| \leq \frac{86|c|}{27}C\sigma^4\|e^{\sigma |D|}u\|_{X_\delta^{0,b}}^5\leq C'\sigma^4\|Iu\|_{X_\delta^{0,b}}^5
\end{align*}
with $C'=\frac{86|c|}{27}C 2^5$. This completes the proof.
\end{proof}

Lemma \ref{lem-app} implies an almost conservation of $E^4_I(t)$ for $t\in [0,\delta]$ when $\sigma$ goes to zero. This together with Lemma \ref{lem-comp} will show that the energy $E^2_I(t)$ is almost conserved.

\begin{corollary}\label{cor-almost}
Let $u\in G_{\delta}^{\sigma,b}$ be the solution of \eqref{equ-1} obtained in Proposition \ref{prop-loc}, $b>\frac{1}{2}$. Assume that $0<\sigma\leq 1$ and  $\|Iu_0\|_{L^2}=\varepsilon_0<1$,  where $I$ is defined by the Fourier symbol $m$ given by \eqref{equ-m}. Then for all $t\in [0,\delta]$
\begin{align}\label{equ-almost-0}
\|Iu(t)\|^2_{L^2}\leq \varepsilon_0^2+\mathcal {O}(\varepsilon_0^3)+C\varepsilon_0^5\sigma^4.
\end{align}
\end{corollary}
\begin{proof}
Since $\|Iu\|_{X_\delta^{0,b}}$ is comparable with $\|u\|_{G_\delta^{\sigma,b}}$, the bound \eqref{equ-local-bound} implies that $\|Iu\|_{X_\delta^{0,b}}\leq C\|Iu_0\|_{L^2}$ for some constant $C>0$. Using the embedding $X_\delta^{0,b} \hookrightarrow L^\infty_tL^2_x$ when $b>\frac{1}{2}$ and  $\|Iu_0\|_{L^2}=\varepsilon_0<1$, we deduce from Lemma \ref{lem-comp} that
\begin{align}\label{equ-almost-1}
E^4_I(0)=E^2_I(0)+\mathcal {O}(\varepsilon_0^3),
\end{align}
and, moreover, for all $t\in (0,\delta]$
\begin{align}\label{equ-almost-2}
E^4_I(t)=E^2_I(t)+\mathcal {O}(\varepsilon_0^3).
\end{align}
Thanks to Lemma \ref{lem-app}, we find for  all $t\in (0,\delta]$
\begin{align}\label{equ-almost-3}
|E^4_I(t)-E^4_I(0)|\leq C\varepsilon_0^5\sigma^4.
\end{align}
Combining \eqref{equ-almost-1}-\eqref{equ-almost-3} implies the desired inequality \eqref{equ-almost-0}.
\end{proof}

\subsection{The proof of Theorem \ref{thm1}}

Let $u_0\in G^{\sigma_0}$ with some $\sigma_0>0$. We can not use the almost conservation law above directly, since the norm $\|u_0\|_{G^{\sigma_0}}$ may be large. To over the difficulty, we need to make a scaling on the solution. Precisely, for every $\lambda>0$, set
$$
u_\lambda(t,x):= \lambda^{-2}u(\frac{t}{\lambda^3}, \frac{x}{\lambda}).
$$
Clearly, $u_\lambda(t,x)$ is also a solution of the KdV equation \eqref{equ-1} on $[0,\lambda^3T]\times \mathbb{R}$ if $u(t,x)$ is a solution on $[0,T]\times \mathbb{R}$. The spatial Fourier transform has the relation
\begin{align}\label{equ-proof-1}
\widehat{u_\lambda}(t, \xi) = \lambda^{-1}\widehat{u}(\frac{t}{\lambda^3}, \lambda \xi).
\end{align}
In particular, we have $\widehat{u_\lambda}(0, \xi) = \lambda^{-1}\widehat{u_0}(\lambda \xi)$. This implies that for all $\sigma>0$
\begin{align}\label{equ-proof-2}
\|u_\lambda(0,\cdot)\|_{G^\sigma}=\lambda^{-\frac{3}{2}}\|u_0\|_{G^{\frac{\sigma}{\lambda}}}.
\end{align}
For every $\varepsilon_0\in (0,1)$, set
\begin{align}\label{equ-proof-3}
\lambda:=\left(1+\frac{\|u_0\|_{G^{\sigma_0}}}{\varepsilon_0}\right)^{\frac{2}{3}}.
\end{align}
Using the embedding $G^{\sigma}\hookrightarrow G^{\frac{\sigma}{\lambda}}$ since $\lambda\geq 1$, and by \eqref{equ-proof-2} we obtain
\begin{align}\label{equ-proof-4}
\|u_\lambda(0,\cdot)\|_{G^{\sigma_0}}\leq \varepsilon_0.
\end{align}
According to Proposition \ref{prop-loc}, problem \eqref{equ-1} has a unique rescaled solution $u_\lambda(t,x)$ with datum $u_\lambda(0,x)$ on the interval $t\in [0,\delta]$, where
\begin{align}\label{equ-proof-4.5}
\delta = \frac{c_0}{(1+\|u_0\|_{G^{\sigma_0}})^{\frac{1}{\frac{3}{4}-b}}}.
\end{align}
Since $\|Iu_\lambda(0,\cdot)\|_{L^2}\leq \|u_\lambda(0,\cdot)\|_{G^{\sigma_0}}\leq \varepsilon_0$, thanks to Corollary \ref{cor-almost}, we obtain for $t\in [0,\delta]$
\begin{align}\label{equ-proof-4.6}
\|u(t)\|_{G^\sigma}\leq 2\|Iu(t)\|_{L^2}\leq 2\sqrt{\varepsilon_0^2+\mathcal {O}(\varepsilon_0^3)+C\varepsilon_0^5\sigma^4}\leq 4\varepsilon_0
\end{align}
where $\sigma=\min\{1,\sigma_0\}$, $\varepsilon_0$ is chosen small enough. Thus $\|u_\lambda(\delta)\|_{G^\sigma}\leq 4\varepsilon_0$. This allows us to take $u_\lambda(\delta)$ as a new data, by virtue of \eqref{equ-proof-4.5}, to obtain a solution on the interval $[\delta, 2\delta]$. Follow this line, by using the local well posedness result and almost conservation law repeatedly, we shall prove that, for arbitrarily large $T$,
\begin{align}\label{equ-proof-5}
\sup_{t\in[0,T]}\|u_\lambda(t)\|_{G^{\sigma(t)}}\leq 4\varepsilon_0,
\end{align}
with for large $t$
\begin{align}\label{equ-proof-6}
\sigma(t)\geq c|t|^{-\frac{1}{4}}.
\end{align}

Now arbitrarily fixed $T$ large. With a little abuse using of notations, we still denote  $E^j_I(t) (j=2,3,4)$ the energies defined in Subsection 2.2 with $u_\lambda$ in place of $u$. Choose $m\in \mathbb{N}$ such that $T\in [m\delta, (m+1)\delta)$. We shall use induction to show for $k=\{1,2,\cdots,m+1\}$ that
\begin{align}\label{equ-proof-7}
\sup_{t\in [0,k\delta]}|E^4_I(t)-E^4_I(0)|\leq Ck\varepsilon_0^5\sigma^4,
\end{align}
\begin{align}\label{equ-proof-8}
\sup_{t\in[0,k\delta]}\|u_\lambda(t)\|_{G^{\sigma}}\leq 4\varepsilon_0.
\end{align}
In fact, for $k=1$, \eqref{equ-proof-7} and \eqref{equ-proof-8} follows from Corollary \ref{cor-almost} and \eqref{equ-proof-4.6}, respectively. Now assume that \eqref{equ-proof-7} and \eqref{equ-proof-8} hold for some $k\in \{1,2,\cdots,m\}$. Take $u_\lambda(k\delta)$ as a new data, by Proposition \ref{prop-loc}, we obtain a solution $u_\lambda$ on the interval $[k\delta, (k+1)\delta]$, and
\begin{align}\label{equ-proof-9}
\sup_{t\in[k\delta, (k+1)\delta]}\|u_\lambda(t)\|_{G^{\sigma}}\leq 4C\varepsilon_0.
\end{align}
Moreover, we apply Corollary \ref{cor-almost} with $u_\lambda$ on the interval $[k\delta, (k+1)\delta]$ to find
\begin{align}\label{equ-proof-10}
\sup_{t\in [k\delta, (k+1)\delta]}|E^4_I(t)-E^4_I(k\delta)|\leq C\varepsilon_0^5\sigma^4.
\end{align}
Combining \eqref{equ-proof-10} and the induction hypothesis \eqref{equ-proof-7}, we obtain
\begin{align}\label{equ-proof-11}
\sup_{t\in [0, (k+1)\delta]}|E^4_I(t)-E^4_I(0)|\leq C(k+1)\varepsilon_0^5\sigma^4.
\end{align}
This proves \eqref{equ-proof-7} with $k$ replaced by $k+1$. Using Lemma \ref{lem-comp}, we deduce from \eqref{equ-proof-11} that
\begin{align}\label{equ-proof-12}
\sup_{t\in [0, (k+1)\delta]}E^2_I(t)\leq \varepsilon_0^2+\mathcal {O}(\varepsilon_0^3)+ C(k+1)\varepsilon_0^5\sigma^4 = \varepsilon_0^2+\mathcal {O}(\varepsilon_0^3)+ C\varepsilon_0^5
\end{align}
provided that
\begin{align}\label{equ-proof-13}
(k+1)\sigma^4=1.
\end{align}
By \eqref{equ-proof-12}, we can choose $\varepsilon_0$ small enough such that
\begin{align}\label{equ-proof-14}
\sup_{t\in [0, (k+1)\delta]}E^2_I(t)\leq 4\varepsilon_0^2.
\end{align}
It follows from \eqref{equ-proof-14} that
$$
\sup_{t\in[0,(k+1)\delta]}\|u_\lambda(t)\|_{G^{\sigma}}\leq 2\sup_{t\in [0, (k+1)\delta]}\sqrt{E^2_I(t)}\leq 4\varepsilon_0.
$$
This proves \eqref{equ-proof-8} with $k$ replaced by $k+1$.

Since $\varepsilon_0\in (0,1)$, we find the lifespan, of local solution, $\delta\sim 1$. Then it follows from \eqref{equ-proof-13} that
\begin{align}\label{equ-proof-15}
\sigma=(k+1)^{-\frac{1}{4}}\geq \left(\frac{T}{\delta}+1\right)^{-\frac{1}{4}}\geq cT^{-\frac{1}{4}},
\end{align}
where $c$ is an absolute constant. Thus, we have proved \eqref{equ-proof-5} and \eqref{equ-proof-6}.

Now we pass the result of $u_\lambda$ to that of $u$. Thanks to \eqref{equ-proof-1}, we have
\begin{align}\label{equ-proof-16}
\widehat{u}(t, \xi) = \lambda\widehat{u_\lambda}(\lambda^3t, \frac{\xi}{\lambda}).
\end{align}
Fixed $T$ arbitrarily large. It follows from \eqref{equ-proof-5}, \eqref{equ-proof-6} and \eqref{equ-proof-16} that
\begin{align}\label{equ-proof-17}
\sup_{t\in[0,T]}\|u(t)\|_{G^{\sigma}}=\lambda^{\frac{3}{2}}\sup_{t\in[0,T]}\|u_\lambda(t)\|_{G^{\lambda^3\sigma}}\leq 4\lambda^\frac{3}{2}\varepsilon_0
\end{align}
with
\begin{align}\label{equ-proof-18}
\sigma \geq \frac{cT^{-\frac{1}{4}}}{\lambda^3}.
\end{align}
By virtue of \eqref{equ-proof-3}, we deduce from \eqref{equ-proof-17}-\eqref{equ-proof-18} that
$$
\sup_{t\in[0,T]}\|u(t)\|_{G^{\sigma}}\leq 4(1+\|u_0\|_{G^{\sigma_0}})
$$
with
$$
\sigma\geq c'T^{-\frac{1}{4}},
$$
where $c'=\frac{c}{(1+\frac{\|u_0\|_{G^{\sigma_0}}}{\varepsilon_0})^2}$. This completes of the proof.

\medskip

\textbf{Acknowledgment}.
The Project was supported by the National Natural Science Foundation of China under grant No. 11701535, and the Natural Science Fund of Hubei
Province under grant No. 2017CFB142.

\end{document}